\documentclass{ws-m3as}

\usepackage[english]{babel}
\usepackage{amsrefs}
\usepackage{amsmath,amssymb,amsfonts,enumerate,textcomp}
\usepackage{url}
\usepackage{graphicx}
\usepackage{xcolor}
\usepackage{wrapfig}

\usepackage[cal=boondox]{mathalfa}

\newcommand{\R}{\mathbb{R}}
\newcommand{\Z}{\mathbb{Z}}

\renewcommand{\leq}{\leqslant}
\renewcommand{\le}{\leqslant}
\renewcommand{\geq}{\geqslant}
\renewcommand{\ge}{\geqslant}

\renewcommand{\epsilon}{\varepsilon}

\newcommand{\e}{\varepsilon}
\newcommand{\const}{\,{\rm const}\,}

\newcommand{\B}{\mathcal{B}}
\newcommand{\al}{\alpha}
\newcommand{\ep}{\epsilon}
\newcommand{\I}{{\mathcal{L}}}
\newcommand{\xs}{\overline{x}}

\begin{document}

\markboth{Serena Dipierro, Stefania Patrizi and Enrico Valdinoci}{Heteroclinic connections for nonlocal equations}

%
\catchline{}{}{}{}{}
%

\title{Heteroclinic connections for nonlocal equations}

\author{Serena Dipierro}

\address{Department of Mathematics and Statistics,
University of Western Australia, \\ 35 Stirling Highway, 
Crawley, WA6009,
Australia\\
serena.dipierro@uwa.edu.au}

\author{Stefania Patrizi}

\address{Department of Mathematics,
University of Texas at Austin, \\
2515 Speedway Stop C1200,
Austin, Texas 78712-1202, USA\\
spatrizi@math.utexas.edu}

\author{Enrico Valdinoci}

\address{Department of Mathematics and Statistics,
University of Western Australia, \\ 35 Stirling Highway, 
Crawley, WA6009,
Australia\\
enrico.valdinoci@uwa.edu.au}

\maketitle

\begin{history}
\received{(Day Month Year)}
\revised{(Day Month Year)}
\accepted{(Day Month Year)}
\comby{(xxxxxxxxxx)}
\end{history}

\begin{abstract}
We construct heteroclinic orbits for a strongly nonlocal
integro-differential equation. Since the energy associated to
the equation
is infinite in such strongly nonlocal regime,
the proof, based on variational methods,
relies on a renormalized energy functional, exploits
a perturbation method
of viscosity type, combined with an auxiliary penalization
method, and develops a free boundary theory for a double obstacle
problem of mixed local and nonlocal type.

The description of the stationary positions for the atom dislocation function in a
perturbed crystal,
as given by the Peierls-Nabarro model, is a particular case of the result presented.
\end{abstract}



\section{Introduction}

Heteroclinic orbits are a classical topic in the context of dynamical systems.
Not only they are trajectories that show an interesting behavior,
providing a connection between two different rest positions,
but they are often the ``building blocks'' for constructing complicated
orbits, drifting from one equilibrium to another, possibly
leading to a chaotic dynamics.

The recent literature has also taken into account the case in which
the ``classical'' differential equations are replaced by integro-differential equations. 
The study of these nonlocal equations
is not only motivated by mathematical curiosity and by
the will driving the scientists of facing with
new challenging problems, but it also possesses concrete motivations in applied
sciences: in particular, our main motivation
for the problem treated in this paper comes from the
description of the stationary positions for the atom dislocation in crystals,
as provided by the Peierls-Nabarro model, see e.g.~\cite{NABA}
and Section~2 of~\cite{MR3296170}. In this context, the evolution
of the dislocation function on the ``slip line'' (i.e., the intersection
between the ``slip plane'', 
along which the crystal experiences a plastic deformation, and a transversal reference plane)
is described by an equation of fractional type,
as a consequence of the balance between
the elastic bonds that link the atoms
and the internal force of the crystals which tends to place all the atoms
into a periodically organized lattice.

Concretely, in the Peierls-Nabarro model for edge dislocations,
one considers equations that can be written along the slip line as
\begin{equation}\label{P2P2} \sqrt{-\Delta}\, Q(x)+W'(Q(x))=0\qquad{\mbox{ for any }}x\in\R,\end{equation}
where~$W$ is a multi-well potential and the diffusion operator
is the square root of the Laplacian, which (up to normalizing multiplicative
constants) is the integro-differential operator
\begin{equation}\label{P2P2P2} \sqrt{-\Delta} \,Q(x):=
{\rm P.V.}\int_\R \frac{Q(x)-Q(y)}{|x-y|^2}\,dy
:=\lim_{\varrho\searrow0}\int_{\R\setminus B_\varrho(x)}\frac{Q(x)-Q(y)}{|x-y|^2}\,dy
.\end{equation}
In the setting of~\eqref{P2P2}, the function~$Q:\R\to\R$ represents
a dislocation function (i.e., roughly speaking, a measure of the atomic
disregistry with respect to the ideal rest configuration of a perfect crystal);
the diffusion operator in~\eqref{P2P2} and~\eqref{P2P2P2}
takes into account the effect on the slip line of the
elastic bonds between different atoms in the crystal
and the potential~$W$ is induced by the large-scale pattern of the crystal itself
(see e.g.~\cite{NABA}
and Section~2 of~\cite{MR3296170} for additional details).
\medskip

We also recall that, at a microscopic scale, there is a strong connection between
the Peierls-Nabarro model and the Frenkel-Kontorova model, see~\cite{MR2852206}.
Related models appear also in the study of 
the Benjamin-Ono equation, see~\cite{MR1442163}, in
boundary reaction equations, see~\cite{MR2177165},
and in spin systems on lattices, see~\cite{MR2219278}.
In addition, the study of nonlocal equations with a singular kernel
is a very intense subject of research in terms of harmonic analysis,
see e.g.~\cite{MR0290095},
and of regularity theory,
see e.g.~\cite{MR2707618}.\medskip

The mathematical framework in which we work here is the following.
Given a function~$Q:\R\to\R$, 
the nonlocal operator that we take into account in this paper is given by
\begin{equation}\label{LOPERATO} \begin{split}
{\mathcal{L}} Q(x)&:= {\rm P.V.}
\int_{\R} \big( Q(x)-Q(y)\big)\,K(x-y)\,dy\\&:=\lim_{\varrho\searrow0}
\int_{\R\setminus B_\varrho(x)} \big( Q(x)-Q(y)\big)\,K(x-y)\,dy
.\end{split}\end{equation}
The kernel~$K$ is supposed to be even and such that
\begin{equation}\label{KERNEL}
\frac{\theta_0}{|r|^{1+2s}} \chi_{[0,r_0]}(r)
\le
K(r)\le \frac{\Theta_0}{|r|^{1+2s}},
\end{equation}
for some~$\Theta_0\ge\theta_0>0$ and some~$r_0>0$, with
\begin{equation}\label{RANGE}
s\in\left( \frac14,\frac12\right].
\end{equation}
Of course, the case under consideration
comprises in particular
the original Peierls-Nabarro model
in~\eqref{P2P2P2}, which corresponds to the choice
\begin{equation}\label{RANGEP2P2}
s:=\frac12\quad{\mbox{ and }}\quad K(r):=\frac{1}{|r|^2}.
\end{equation}
In the equations that we take into account, the diffusive operator~${\mathcal{L}}$ is balanced
by a forcing term of potential type. More precisely,
we consider a non-negative multi-well potential~$W\in C^2(\R,\R)$
with a locally finite set of minima.
Namely, we suppose that~$W\ge0$ and that there
exists~${\mathcal{Z}}\subset\R$
which is a discrete set (i.e., it has no accumulation points)
such that
\begin{equation}\label{POT:GROW0} W(\zeta)=0 {\mbox{ for any }} \zeta\in{\mathcal{Z}}
{\mbox{ and }}W(r)>0 {\mbox{ for any }} r\in\R\setminus{\mathcal{Z}}.\end{equation}
We also suppose that~$W$ grows quadratically from its minima, that is
\begin{equation}\label{POT:GROW}\begin{split}
& c_0 |\xi|^2\le W(\zeta+\xi)\le C_0 |\xi|^2,
\\&{\mbox{for some~$C_0>c_0>0$,
for all~$\zeta\in{\mathcal{Z}}$ and~$\xi\in B_{\delta_0}$, with~$\delta_0>0$.}}\end{split}\end{equation}

We also assume that $W$ is monotone near the equilibria, that is
\begin{equation}\label{AGGchew}
{\mbox{$W'(x)>0$ for all~$x\in(\zeta,\zeta+\delta_0)$,
and $W'(x)<0$ for all~$x\in(\zeta-\delta_0,\zeta)$.}}
\end{equation}
\medskip

In our framework, the potential is modulated by an oscillatory function~$a$.
This function is supposed to maintain the sign of the potential, namely we assume that
\begin{equation} \label{POT:GROW2} a(x)\in [\underline a,\overline a]\quad{\mbox{ for all }}x\in\R,\end{equation}
for some~$\overline a> \underline a>0$.

We also assume that~$a$ is ``non-degenerate''. More precisely,
we suppose that there exist~$m_1$, $m_2\in\R$ and $\omega$, $\theta>0$ such that
\begin{equation}\label{m1m2theta}
m_2-m_1\geq 2\omega+\theta,
\end{equation} and, for $i\in\{1,2\}$,  
\begin{equation}\label{a nondegenerate}\begin{split}
a(x)-a(x-\theta)\ge\gamma\quad\text{and}\quad a(x)-a(x+\theta)\ge\gamma,\quad\text{for all }x\in[m_i-\omega,m_i+\omega],
\end{split}
\end{equation}
for\footnote{For concreteness,\label{7UJSjjasdfoor}
we mention that the function
$$ a(x):=2+\e\cos(\delta x)$$
with~$\e$, $\delta\in(0,1]$ satisfies~\eqref{a nondegenerate}
with~$m_1:=0$, $m_2:=\frac{2\pi}\delta$,
$\omega:=\frac{\pi}{4\delta}$, $\theta:=\frac\pi\delta$
and~$\gamma:={\sqrt{2}}\,\e$. Indeed,
in this case, 
\begin{eqnarray*}
&&\inf_{x\in[m_1-\omega,m_1+\omega]\cup
[m_2-\omega,m_2+\omega]} a(x)-a(x\pm\theta)\\
&=& \inf_{x\in[-\frac{\pi}{4\delta},\frac{\pi}{4\delta}]\cup
[\frac{2\pi}\delta-\frac{\pi}{4\delta},\frac{2\pi}\delta+\frac{\pi}{4\delta}]}
\e\big( \cos(\delta x)-\cos(\delta x\pm\delta\theta)\big)\\
&=& \inf_{y\in[-\frac\pi4,\frac\pi4]}
\e\big( \cos {y}-\cos(y\pm\pi)\big)\\&=&2
\inf_{y\in[-\frac\pi4,\frac\pi4]}
\e \cos {y}\\
&=& \sqrt{2}\,\e.
\end{eqnarray*}
This example shows that there exist ``small and slow perturbations 
of constant functions'' that satisfy~\eqref{a nondegenerate}.}
some~$\gamma>0$.\medskip

In this setting, the equation that we study here
has the form
\begin{equation}\label{EQUAZIONE}
{\mathcal{L}}Q^\star(x)+a(x)\,W'\big(Q^\star(x)\big)=0
\quad{\mbox{ for all }}x\in\R.
\end{equation}
Of course, when~${\mathcal{L}}$ is
replaced by the classical
second order differential operator, equation~\eqref{EQUAZIONE} may be seen
as a pendulum-like equation. 
\medskip

The main objective of this paper is to construct heteroclinic solutions
of~\eqref{EQUAZIONE}, i.e. orbits which connect two different
equilibria.
The equilibria that we connect are the ``closest''
possible, as specified by the following definition:

\begin{definition}\label{PRIMICI}
We say that~$\zeta_1\ne\zeta_2\in{\mathcal{Z}}$
are nearest neighbors if the open segment joining~$\zeta_1$
and~$\zeta_2$ does not contain any point of~${\mathcal{Z}}$.
\end{definition}

Then, in this setting, our main result on the existence
of heteroclinics for equation~\eqref{EQUAZIONE} is the following:

\begin{theorem}\label{MAIN}
Let~$\zeta_1\in{\mathcal{Z}}$ and let~$\zeta_2\in{\mathcal{Z}}\setminus
\{\zeta_1\}$ be a nearest neighbor of~$\zeta_1$.

Then, there exists a solution~$Q^\star:\R\to\big[\min\{\zeta_1,\zeta_2\},
\max\{\zeta_1,\zeta_2\}\big]$ of~\eqref{EQUAZIONE} such that
\begin{equation}\label{0LIMI}
\lim_{x\to-\infty}Q^\star(x)=\zeta_1\; \qquad{\mbox{ and }}\qquad
\;\lim_{x\to+\infty}Q^\star(x)=\zeta_2.
\end{equation}
\end{theorem}

{F}rom the physical point of view, Theorem~\ref{MAIN}
states that an arbitrarily small perturbation of the global crystal configuration (as induced
by the perturbation function~$a$) can lead to a nontrivial displacements
of the atoms of the crystal.
\medskip

We observe that Theorem~\ref{MAIN}
is new even in the model case of the square root of the Laplacian
(as described by~\eqref{P2P2P2}
and~\eqref{RANGEP2P2}).
In particular, Theorem~\ref{MAIN} is new even for the original perturbed
Peierls-Nabarro model. The proof that we perform is more general and it comprises
several other cases of interest, and it is rather complicated, but it does not
enjoy significant simplifications in the particular case of Peierls-Nabarro,
and this is the main reason for which we deal with equation~\eqref{EQUAZIONE}
here in its full generality.
\medskip

The advantage of keeping such a generality is that equation~\eqref{EQUAZIONE}
also applies to models which go beyond atom dislocation in crystals.
For instance, similar equations appear in the dynamics of biological populations.
In this case a logistic nonlinearity is coupled to a diffusive regime,
and several experiments have confirmed the tendency of
many biological species to exploit nonlocal and fractional diffusive strategies,
see~\cites{VIS, HUM}. 
The experiments on the different biological species have been also
fruitfully compared and confirmed by numerical simulations, see~\cite{MR3272943}.
{F}rom the mathematical point of view, anomalous diffusion
can provide concrete advantages for the structural organization of a biological
species in terms of resources consumption (see~\cites{MR2556498, MR3590678, MR3579567}),
it can be related to the optimal distribution of predators
in an environment with sparse prays (see~\cite{MR3590646}),
and can exhibit different stability patterns (see~\cites{MR2897881, MR3026598}).
Related nonlocal equations also naturally appear in genetics, see e.g.~\cite{MR1636644}
and the references therein.\medskip

Hence, in view of its fascinating mathematical framework, 
and due to the great source of concrete applications
in physics, material sciences, ethology and genetics,
we think it is worth to consider
equation~\eqref{EQUAZIONE} in its generality, since
the analytical difficulties do not get simplified by
reducing to the classical Peierls-Nabarro
model, and the additional generality allows us to comprise many cases
of concrete interest in applied sciences.
\medskip

We also point out that, differently from the classical case, the asymptotic
expression in~\eqref{0LIMI}
is not an immediate consequence of fractional
energy estimates since,
when~$s\in\left(0,\frac12\right]$, functions in~$H^s(\R)$ are not
necessarily infinitesimal at infinity
(see e.g. Appendix~\ref{APPB} for a simple example of this important
phenomenon).\medskip

We remark
that, on the one hand, the non-degeneracy condition in~\eqref{a nondegenerate}
introduces an additional difficulty to the problem, making it {\em
not invariant under translations}. 
On the other hand, the existence
of heteroclinic solutions in the homogeneous setting
is known at least in some cases, see~\cites{MR2177165, MR3081641, MR3280032, MR3460227, MURA}.
{F}rom the technical point of
view, the function~$a$ produces
an important difference with respect to the previous
literature on the subject, since the translation invariance of the system
implies the monotonicity of the heteroclinic, which
in turn implies a series of analytic estimates on the energy functional
and allows one to use of more direct minimization principles (see~\cites{MR3081641, MR3280032, MURA}
for further details).
\medskip

In general, our perspective is that not only
the non-degeneracy condition in~\eqref{a nondegenerate} provides
a useful technical tool for the construction of
heteroclinic solutions in problems that are
not invariant under translations, but also it is an essential
ingredient to deal with homoclinic and multibump orbits.
As a matter of fact, as a future project we aim
at using the heteroclinic connections constructed here
as a ``building block'' for these more general types of trajectories
(and, for instance, homoclinic solutions do not
exist when~$a$ is constant, see~\cite{MR3280032}
in case of even solutions, and~\cite{2019arXiv190701491C}
for the general case, therefore condition~\eqref{a nondegenerate}
cannot be removed when dealing with homoclinic connections).
\medskip

{F}rom the physical point of view,
thinking about dislocation models, the function~$a$
can be considered as an arbitrarily small perturbation (and with
oscillations having arbitrarily small frequencies, see the footnote on page~\pageref{7UJSjjasdfoor}) of the structural potential~$W$
which is induced by the ``large scale'' pattern of the crystal.
In a sense, one can think that an appropriate external force
can slightly, and suitably, modify the perfect periodic structure of the material,
thus producing this perturbed potential configuration.
\medskip

As a matter of fact, on the one hand Theorem~\ref{MAIN}
is not necessarily a perturbation result, since
it also comprises cases in which the function~$a$
is not a perturbation of the constant; on the other hand,
Theorem~\ref{MAIN} includes the perturbative setting
as a special case, and, in this sense, we can
also consider Theorem~\ref{MAIN} as a ``structural stability''
result which ensures the conservation of heteroclinic orbits
under a (suitably nondegenerate) perturbation.
\medskip

Furthermore, the construction of heteroclinic orbits for
ordinary differential equations is a well-studied topic in the literature
and, in this sense, Theorem~\ref{MAIN} here is a nonlocal counterpart
of some of the celebrated results obtained
in~\cites{MR1030854, MR1304144, MR1799055, MR1804957}
for ordinary differential equations
and Hamiltonian systems. Of course, the case of nonlocal equations
is conceptually quite different from that of ordinary differential equations,
since usual ``glueing'' and ``cut-and-paste'' methods are not available,
due to far-away energy interactions. We refer to~\cite{MR3469920}
for a general introduction to nonlocal problems, also motivated
from water wave models, phase transitions, material sciences and biology.\medskip

A result similar to Theorem~\ref{MAIN} when the nonlocal parameter~$s$
lies in the range~$\left(\frac12,1\right)$ has been obtained in~\cite{MR3594365}, see also~\cite{MR3625078}.
In case of  homogeneous media (i.e., when~$a$ is constant),
heteroclinic connections corresponding to parameter
ranges~$s\in\left(0,\frac12\right]$ have been 
studied in~\cites{MR3081641, MR3280032, MURA} by energy renormalization methods.\medskip

Concerning the nonlocal parameter range considered in this paper,
we recall that the case~$s\in\left(0,\frac12\right)$ can present several
technical and conceptual differences with respect to the case~$s\in\left(\frac12,1\right)$
(the case~$s=\frac12$ being typically ``in between'' the two cases).
For instance, as shown in~\cites{MR2564467, MR2948285}, several
fractional equations corresponding to the parameter range~$s\in\left[\frac12,1\right)$
present a ``local behavior'' at a large scale,
while they preserve a ``nonlocal behavior'' at any scale when~$s\in
\left(0,\frac12\right)$.
\medskip

In our setting, to deal with the case~$s\in\left(\frac14,\frac12\right]$ we
will adopt a strategy that has been also very recently used in~\cite{MURA}
and based on two basic steps:
\begin{itemize}
\item Given~$\zeta_1$, $\zeta_2\in{\mathcal{Z}}$,
we take a function
\begin{equation}\label{TRAQQ}
Q_{\zeta_1,\zeta_2}^\sharp\in C^\infty\big(\R, (\min\{\zeta_1,\zeta_2\},\max\{\zeta_1,\zeta_2\})\big)\end{equation} to be such that~$Q_{\zeta_1,\zeta_2}^\sharp(x)=\zeta_1$
for any~$x\in\left(-\infty,-1\right)$
and~$Q_{\zeta_1,\zeta_2}^\sharp(x)=\zeta_2$
for any~$x\in\left(1,+\infty\right)$.
Then, we consider a renormalized energy functional of the form
\begin{eqnarray*}&&
I_0(Q):=
\int_\R a(x)W\big(Q(x)\big)\,dx\\ &&\;\;
+\frac14
\iint_{\R\times \R}\Big( \big| Q(x)-Q(y)\big|^2
-\big| Q_{\zeta_1,\zeta_2}^\sharp(x)-Q_{\zeta_1,\zeta_2}^\sharp(y)\big|^2\Big)
\,K(x-y)\,dx\,dy.\end{eqnarray*}
The device of dealing with a renormalized energy is needed
in order to avoid the divergence of the energy due to nonlocal effects in this parameter range.
We stress that this energy divergence is unavoidable, since, for instance,
one can easily check that the fractional Sobolev (or 
Aronszajn-Gagliardo-Slobodeckij)
seminorm in~$H^s(-R,R)$ of a smooth function
connecting two constants goes like~$\log R$ when~$s=\frac12$,
and like~$R^{1-2s}$ when~$s\in\left(0,\frac12\right)$, thus diverging as~$R\to+\infty$.
\item We will perturb the original energy functional by a classical Dirichlet energy.
This step is very convenient, since it allows to deal with continuous trajectory in a
perturbed setting (notice that, when~$s\in\left(0,\frac12\right]$, functions in~$H^s(\R)$
are not necessarily continuous, see e.g.
Appendix~\ref{APPB} for a simple example). After dealing with a minimization argument
for such perturbed energy functional, we will obtain uniform estimates that
will allow us to pass to the limit.
\end{itemize}

Furthermore, in our setting,
we exploit an additional penalization
method in~$L^2(\R)$, which has the technical advantage
of localizing the transition sufficiently close to a prescribed
position (say, the origin).\medskip

A series of analytical techniques coming from elliptic partial differential equations
are also crucially exploited in our proofs:
\begin{itemize}
\item We will make use of viscosity solution methods in order to obtain
regularity theories that are uniform in the perturbation parameter related to the
Dirichlet energy (this is a fundamental step in order to ``remove''
the ``local and elliptic energy perturbation'' in the limit).
\item We will study a double obstacle problem of mixed local and nonlocal type,
which arises from the constrained minimization of the energy functional
(this step is crucial in order to estimate ``how the orbits separates from the constraints''). 
\item The estimates obtained on the penalized equation
will then be sufficiently stable to remove the additional
penalization term.
\end{itemize}
In general, we believe that a very interesting feature provided by the equations
related to the Peierls-Nabarro model lies in the fact that their complete understanding
requires a {\em synergic combination of resources and methods coming from different
specific backgrounds}, which include, among the others,
mathematical physics, calculus of variations,
partial differential equations, free boundary problems, geometric measure theory, harmonic analysis
and the theory of pseudodifferential operators.\medskip

The parameter range considered in this paper
has also a special energy feature. Namely,
while the interaction energy of fractional Sobolev type
of a heteroclinic connection is divergent,
the part coming from the potential is typically finite
under assumption~\eqref{RANGE}. To check this,
we recall formula~(12) in~\cite{MR3081641}, according to
which a heteroclinic orbit~$Q(x)$ converges to the equilibrium
in the homogeneous case like~$\frac\const{1+|x|^{2s}}$.
Since, by~\eqref{POT:GROW},
the potential~$W$ is quadratic near the equilibria, the potential
energy term of such trajectory behaves like
$$ \int_\R \frac\const{(1+|x|^{2s})^2}\,dx,$$
which is finite when~$s$ lies above the threshold~$1/4$.\medskip

For this reason, when~$s$ lies below~$1/4$, it could be expected
that a second energy renormalization is needed in order to
apply variational methods (e.g. in the approach given by
formula~(13) in~\cite{MR3081641}) and we plan to explore this
parameter range in future works.\medskip

The rest of the paper is organized as follows.
In Section~\ref{S:1}, we fix some notation, to be used in the rest of the paper.
In Section~\ref{S:2}, we give two elementary proofs
establishing a uniform bound for a nonlocal equation and
a regularity result for a perturbed problem
(in our setting,
such bound is important to obtain uniform
estimates in a perturbed problem, and the regularity result is
useful to estimate errors in the ``cut-and-paste'' procedures).

The proof of Theorem~\ref{MAIN} is then developed in
Sections~\ref{S:3}--\ref{UPAlao}. More precisely,
Section~\ref{S:3} is devoted to an energy estimate from below. In our
setting, this bound is important 
to guarantee
the necessary compactness for the direct methods of the calculus of variations.

Then, we exploit these variational methods to construct
the heteroclinic connections, by proceeding step by step.
First, in Section~\ref{S:4}
we consider a constrained and perturbed problem.
The additional perturbation provides the technical advantage that
all the orbits with finite energy are in fact continuous,
and this fact will allow us to make use of geometric arguments
in the analysis of such orbits. The constrain is also useful to ``force''
the orbits close to the equilibria at infinity.
As a matter of fact, in Section~\ref{sec:LW},
using a double obstacle problem approach,
we show that constrained minimizers are continuous with
uniform bounds.

Interestingly, this obstacle problem is also of mixed local and nonlocal type,
and this is a class of problems rarely studied in the existing literature.
For our goals, the achievement
of uniform estimates for this problem is crucial in order to have precise information
when the orbit touches the variational constraints.

Also, in Sections~\ref{sec:CI} and~\ref{MI100} we recall the notions
of {\em{clean intervals}} and {\em{clean points}}, and we prove
some stickiness properties of the energy minimizers.

In Section~\ref{REDUSE}, we reduce the proof
of our main result to the case in which the set of equilibria
consists of two points.
Then, in Section~\ref{UNCO},
by taking the asymptotic constraints ``far enough'', we will
produce a free, i.e. unconstrained, minimizer. 

In Section~\ref{S023fhYU0193933}, by using estimates
that are uniform with respect to the viscous perturbative parameter,
we will be able to remove the viscosity perturbation.

Interestingly, the whole variational scheme presented so far
is improved by an additional penalization term. This technical device
is convenient to ``localize'' the transition of the minimal heteroclinic:
to obtain the desired solution, one needs to remove 
the penalization term. To this end, one
needs to obtain spatial asymptotics at infinity, 
in order to detect the limit behavior of the desired heteroclinics.
For this, it is also useful
to observe that solutions
that are asymptotically confined near an equilibrium are
in fact converging to it: this observation is contained in Section~\ref{ASynsdagg}.

With this, the penalized term in the equation 
will be removed in
Section~\ref{UPAlao}, thus completing the proof of Theorem~\ref{MAIN}.

\section{Notation}\label{S:1}

\begin{itemize}
\item 
Given~$I$, $J\subseteq\R$ and~$f$, $g:\R\to\R$, we set
\begin{equation}\label{0q0303030303s} \B_{I,J}(f,g):=
\iint_{I\times J} \big( f(x)-f(y)\big)\big(g(x)-g(y)\big)\,K(x-y)\,dx\,dy,\end{equation}
and 
\begin{equation}\label{Eenergy}
E_{I\times J}(f):=\iint_{I\times J}\Big( \big| f(x)-f(y)\big|^2
-\big| Q_{\zeta_1,\zeta_2}^\sharp(x)-Q_{\zeta_1,\zeta_2}^\sharp(y)\big|^2\Big)
\,K(x-y)\,dx\,dy.
\end{equation}
Notice that
\begin{equation}\label{BIJJI}\begin{split}&
\B_{J,I}(f,g)=
\iint_{J\times I} \big( f(x)-f(y)\big)\big(g(x)-g(y)\big)\,K(x-y)\,dx\,dy\\&\qquad=
\iint_{I\times J} \big( f(y)-f(x)\big)\big(g(y)-g(x)\big)\,K(y-x)\,dy\,dx=\B_{I,J}(f,g),\end{split}\end{equation}
since~$K$ is even.
Similarly,
$$E_{I\times J}(f)=E_{J\times I}(f).$$
We will also use the notation
$$ E_{I^2}(f)= E_{I\times I}(f).$$
\item Given~$X$, $Y\subseteq\R$, we set
\begin{equation}\label{NORM}
[v]_{K,X\times Y}:=
\sqrt{\iint_{X\times Y} \big| v(x)-v(y)\big|^2\,K(x-y)\,dx\,dy}\,.\end{equation}
\item
The Lebesgue measure of a set~$A$ will be denoted by~$|A|$.
\end{itemize}

\section{A uniform bound and a regularity result for a nonlocal equation}\label{S:2}

In this section, ${\mathcal{L}} u$ denotes
the differential operator defined on smooth bounded functions
as follows
\begin{equation}\label{LOPERATOND}\begin{split}
{\mathcal{L}} u(x):=\;& {\rm P.V.}
\int_{\R^N} \big( u(x)-u(y)\big)\,K(x-y)\,dy\\:=\;&\lim_{\varrho\searrow0}
\int_{\R^N\setminus B_\varrho(x)} \big( u(x)-u(y)\big)\,K(x-y)\,dy,
\end{split}\end{equation} where $K$ is an even kernel such that 
\begin{equation*}
\frac{\theta_0}{|r|^{N+2s}} \chi_{[0,r_0]}(r)
\le
K(r)\le \frac{\Theta_0}{|r|^{N+2s}},
\end{equation*}
for some~$\Theta_0\ge\theta_0>0$ and some~$r_0>0$, with $s\in (0,1)$.
Of course, the setting in~\eqref{LOPERATO} is comprised here with~$N:=1$.

The next is a uniform
regularity result 
dealing with a perturbed problem:

\begin{lemma}\label{etaregularitylemmaND}
Let~$\eta\in[0,1],\,s\in(0,1)$ and~$f_1,\,f_2\in \R$. Let $u\in L^\infty(\R^N)\cap C({ {B_1}})$ be a  viscosity subsolution of 
\begin{equation}\label{eq1283ieNDsub}
-\eta\Delta u+{\mathcal{L}}u+f_1=0 \quad{\mbox{ in }}\;{ {B_1}},\end{equation}
and a  viscosity supersolution of 
\begin{equation}\label{eq1283ieNDsuper}
-\eta\Delta u+{\mathcal{L}}u+f_2=0 \quad{\mbox{ in }}\; { {B_1}}.\end{equation}
Then, $u\in C^{0,\alpha}( { B_{1/2} })$ for any $\alpha<\min\{2s,1\}$ and 
\begin{equation} \label{mainholdest}[u]_{C^{0,\alpha}({ B_{1/2} })}\leq
C\left(f_2-f_1+\|u\|_{L^\infty(\R^N)}\right)^\frac{\alpha}{2s}\|u\|_{L^\infty(\R^N)}^{1-\frac{\alpha}{2s}},
\end{equation}
for some $C>0$ independent of $\eta$.
\end{lemma}

\begin{proof}
We use appropriate
techniques from the theory of regularity of viscosity solutions of uniformly elliptic  second-order local operators,
see~\cite{il}, and   recently extended 
to nonlocal operators, see e.g.~\cites{MR2735074,MR2946964}, adapted to our context.
Let us introduce the following notation: given~$r>0$,
for a function $\phi$ we define
$$\I^{1,r}\phi(x):=\int_{\{ {|z|}\leq r\}}(\phi(x)-\phi(x+z)+
\nabla \phi(x)\cdot z)K(z)\,dz$$ and 
$$\I^{2,r}\phi(x):=\int_{\{{|z|}\geq r\}}(\phi(x)-\phi(x+z))K(z)\,dz.$$ 
Then,
\begin{equation}\label{3.9bis}
\I \phi(x)=\I^{1,r}\phi(x)+\I^{2,r}\phi(x).\end{equation}
We let $\phi\in C^\infty(\R^N;\R_0^+)\cap W^{2,\infty}(\R^N)$ be  such that   $\phi(x)=0$ for all $x\in { B_{1/2} }$ and $\phi(x)\geq 1$ for all $x\in { { {\R^N}\setminus B_{3/4} }}$.
We then define 
\begin{equation}\label{psiholdthmdef}\psi(x):=2\|u\|_{L^\infty (\R^N)}\phi(x).\end{equation}
Since $\phi\equiv 0$ in ${ B_{1/2} }$, to prove that $u\in C^{0,\alpha}({ B_{1/2} })$ for any $\alpha<2s$,  it is enough to show 
that given any~$\alpha<2s$, with~$\alpha\in(0,1)$, 
there exists $L>0$ such that, for all $x_1,x_2\in\R^N$, 
\begin{equation}\label{holderplupsi}
u(x_1)-u(x_2)- L|x_1-x_2|^\alpha-\psi(x_1)\leq 0.
\end{equation}
We argue by contradiction, assuming that \eqref{holderplupsi} does not hold  true. 
For $\ep>0$, let $u^{\ep}$ and $u_{\ep}$
be respectively the  sup and inf convolution of
$u$ in $\R^N$, i.e.,
\begin{eqnarray*}
u^{\ep}(x)&:=&\sup_{y\in \R^N}\left(u(y)-\frac{1}{2\ep}|x-y|^2\right)\\
{\mbox{and }}\quad u_{\ep}(x)&:=&\inf_{y\in \R^N}\left(u(y)
+\frac{1}{2\ep}|x-y|^2\right).\end{eqnarray*}
We notice that
\begin{equation}\label{3.9ter}
u^{\epsilon}(x)\ge u(x) \ge u_{\epsilon}(x).
\end{equation}
Moreover,
$u^{\ep}$ is  semiconvex and is a subsolution of \eqref{eq1283ieNDsub} in $B_{{1}-\rho}$
and~$u_{\ep}$ is semiconcave and is a supersolution
of~\eqref{eq1283ieNDsuper} in $B_{{1}-\rho}$, for some $\rho=\rho(\ep)>0$,
which is infinitesimal as~$\e\searrow0$, 
see e.g.
Proposition~III.2 in~\cite{MR1116854}.

Since  \eqref{holderplupsi} does not hold true, there exists $\al\in(0,2s)$
such that, for
any~$L>0$ and~$\ep>0$,
\begin{eqnarray*}&&
\sup_{(x_1,x_2)\in \R^{2N}}u^{\ep}(x_1)-u_{\ep}(x_2)-L|x_1-x_2|^\al-\psi(x_1)\\&
 \geq& \sup_{(x_1,x_2)\in\R^{2N}}u(x_1)-u(x_2)-L|x_1-x_2|^\al-\psi(x_1)>0,
 \end{eqnarray*}
where we also used~\eqref{3.9ter}.
Then, for any $L>0$ and
$\ep>0$, 
the supremum on $\R^{2N}$ of the function 
\begin{equation}\label{newlabellemmhold1}u^{\ep}(x_1)-u_{\ep}(x_2)-L|x_1-x_2|^\al-\psi(x_1)\end{equation}
is positive and is attained at some point~$(\xs_1,\xs_2)\in  \R^{2N}$.
Moreover, for $\ep$ small enough, we have that~$\xs_1\neq\xs_2$. 
We remark that
\begin{equation}\label{eq::s11}
|\xs_1-\xs_2|\leq\left(\frac{2\|u\|_{L^\infty(\R^N)}}{L}\right)^\frac{1}{\al}.
\end{equation}
Using that  $\phi\ge 1$ in $ {\R^N\setminus B_{3/4} }$ and  \eqref{psiholdthmdef},
we see that, for all $x_1\in {\R^N\setminus B_{3/4} }$,
$$u^{\ep}(x_1)-u_{\ep}(x_2)-L|x_1-x_2|^\al-\psi(x_1)\leq u^{\ep}(x_1)-u_{\ep}(x_2)-2\|u\|_{L^\infty(\R^N)}\le o_\ep(1),$$
where $o_\ep(1)\to 0$ as $\ep\searrow0$. 
Thus we must have $x_1\in { B_{3/4} }$ for $\ep$ small enough, and by \eqref{eq::s11},
if \begin{equation*}
L\ge16\|u\|_{L^\infty(\R^N)},\end{equation*}
 we have that  
 \begin{equation}\label{x2x2assinholreg}\xs_1,\xs_2\in { B_{7/8} }.\end{equation}
The function in \eqref{newlabellemmhold1} is semiconvex,
hence, by Aleksandrov's Theorem, twice differentiable almost eve\-ry\-where. 
Let us now introduce a perturbation of it, for which we
can choose maximum points of twice differentiability. 
 
First
we transform $(\xs_1,\xs_2)$ into a strict
maximum point. In order to do that, we consider a smooth function
$h:\R^+\rightarrow\R$, with compact support, such that
$h(0)=0$ and $h(t)>0$ for $0<t<1$, we fix a small $\beta>0$ and we set
$$\theta(x_1,x_2):=\beta h(|x_1-\xs_1|^2)+\beta h(|x_2-\xs_2|^2).$$
Clearly, $(\xs_1,\xs_2)$ is a strict maximum point of
$u^{\ep}(x_1)-u_{\ep}(x_2)-L|x_1-x_2|^\alpha-\psi(x_1)-\theta(x_1,x_2)$.

Next we consider a smooth function $\tau:\R^N\rightarrow\R$ such
that $\tau(x)=1$ if $|x|\leq1/2$ and $\tau(x)=0$ for $|x|\geq1$.
By Jensen's Lemma, see e.g. Lemma~A.3 of~\cite{MR1118699},
for every small and positive $\delta$, there
exist $q_1^\delta,\,q_2^\delta\in\R^N$ with
$|q_1^\delta|\,,|q_2^\delta|\leq \delta$,
such that the function
\begin{equation}\label{jensen}\Phi(x_1,x_2):=u^{\ep}(x_1)-u_{\ep}(x_2)
-L|x_1-x_2|^\al-\varphi_1(x_1)-\varphi_2(x_2),\end{equation}
where
\begin{eqnarray*}
\varphi_1(x_1)&:=&\psi(x_1)+\beta h(|x_1-\xs_1|^2)+
\tau(x_1-\xs_1)q_1^\delta\cdot x_1,\\
{\mbox{and }}\quad \varphi_2(x_2)&:=&\beta h(|x_2-\xs_2|^2)+
\tau(x_2-\xs_2)q_2^\delta\cdot x_2,\end{eqnarray*}
has a maximum at $(x_1^\delta,x_2^\delta)$, with
\begin{equation}\label{eq::s12}
|x_1^\delta-\xs_1|,\,|x_2^\delta-\xs_2|\leq\delta
\end{equation}
and $u^{\ep}(x_1)-u_{\ep}(x_2)$ is twice
differentiable at $(x_1^\delta,x_2^\delta)$. In
particular, $u^{\ep}$ is twice differentiable with respect
to~$x_1$ at~$x_1^\delta$ and~$u_{\ep}$ is twice differentiable
with respect to~$x_2$ at~$x_2^\delta$.

We remark that the function $\tau$ has been introduced to
make $\I^{2,r}\varphi_1$ and $\I^{2,r}\varphi_2$ finite.
Also, for $\delta$ small
enough, and $\rho$ small enough, by \eqref{x2x2assinholreg} and \eqref{eq::s12}, 
 we have that
 \begin{equation}\label{x1x2intherightset}x_1^\delta,x_2^\delta\in B_{1-\rho},\end{equation}
 and that~$x_1^\delta\neq x_2^\delta$. In particular, 
 this will allow us to compute the derivatives of the function in~\eqref{jensen}.

Since
$(x_1^\delta,x_2^\delta)$ is a maximum point for~$\Phi$, we have 
\begin{equation}\label{gradientzerohold}\begin{split}
\nabla u^{\ep}(x_1^\delta)&=\nabla\varphi_1(x_1^\delta)+\al
L|x_1^\delta-x_2^\delta|^{\al-2}(x_1^\delta-x_2^\delta)\\{\mbox{and }}\quad
\nabla u_{\ep}(x_2^\delta)&=-\nabla\varphi_2(x_2^\delta)+\al
L|x_1^\delta-x_2^\delta|^{\al-2}(x_1^\delta-x_2^\delta).\end{split}\end{equation}
Moreover the inequalities
\begin{eqnarray*}
&& \Phi(x_1^\delta+z,x_2^\delta)\leq 
\Phi(x_1^\delta,x_2^\delta),\\
&&\Phi(x_1^\delta,x_2^\delta+z)\leq 
\Phi(x_1^\delta,x_2^\delta)\\
{\mbox{and }}&&
\Phi(x_1^\delta+z,x_2^\delta+z)\leq 
\Phi(x_1^\delta,x_2^\delta),\end{eqnarray*}
for any $z\in\R^N$, together with~\eqref{gradientzerohold}, give respectively: 
\begin{equation}\label{holdthmdis2}\begin{split}
&u^{\ep}(x_1^\delta+z)-u^{\ep}(x_1^\delta)
-\nabla u^{\ep}(x_1^\delta)\cdot z\\ \leq \;&
\varphi_1(x_1^\delta+z)-\varphi_1(x_1^\delta)
-\nabla\varphi_1(x_1^\delta)\cdot z\\&\quad 
+L|x_1^\delta+z-x_2^\delta|^\al-L|x_1^\delta-x_2^\delta|^\al-\al
L|x_1^\delta-x_2^\delta|^{\al-2}(x_1^\delta-x_2^\delta)\cdot z,
\end{split}\end{equation}
and
\begin{equation}\label{holdthmdis3}\begin{split}&-(u_{\ep}
(x_2^\delta+z)-u_{\ep}(x_2^\delta)-\nabla u_{\ep}(x_2^\delta)\cdot
z)\\ \leq\;&
\varphi_2(x_2^\delta+z)-\varphi_2(x_2^\delta)-\nabla\varphi_2(x_2^\delta)\cdot
z\\&\quad +L|x_1^\delta-z-x_2^\delta|^\al-L|x_1^\delta-x_2^\delta|^\al+\al
L|x_1^\delta-x_2^\delta|^{\al-2}(x_1^\delta-x_2^\delta)\cdot z,
\end{split}\end{equation}
and, for any $r>0$,
\begin{equation}\label{holdthmdis}\begin{split}
&u^{\ep}(x_1^\delta+z)-u^{\ep}(x_1^\delta)
-{\chi}_{B_{r}}(z)\nabla u^{\ep}(x_1^\delta)\cdot z
\\ \leq\;& u_{\ep}
(x_2^\delta+z)-u_{\ep}(x_2^\delta)-{\chi}_{B_{r}}(z)\nabla u_{\ep}(x_2^\delta)\cdot
z\\&\quad+\varphi_1(x_1^\delta+z)-\varphi_1(x_1^\delta)
-{\chi}_{B_{r}}(z)\nabla \varphi_1(x_1^\delta)\cdot z
\\&\quad
+\varphi_2(x_2^\delta+z)-\varphi_2(x_2^\delta)-{\chi}_{B_{r}}(z)\nabla\varphi_2(x_2^\delta)\cdot
z.\end{split}\end{equation} 
The last inequality in particular implies that
\begin{equation}\label{holdthmdis4}\I^{2,r}u^{\ep}(x_1^\delta){\geq}
\I^{2,r} u_{\ep}(x_2^\delta)
+\I^{2,r}\varphi_1(x_1^\delta)+\I^{2,r}\varphi_2(x_2^\delta),\end{equation} and
\begin{equation}\label{holdthmdisbis}
D^2u^{\epsilon}(x_1^\delta)-D^2u_{\epsilon}
(x_2^\delta)\le C(\beta  +\|u\|_{L^\infty(\R^N)})I_N,\end{equation}
where $I_N$ is the $N\times N$ identity matrix. Here and henceforth $C$ denotes various
positive constants independent of the pa\-ra\-me\-ters.

Now, using that $u^\ep$ and $u_\ep$ are respectively subsolution
of~\eqref{eq1283ieNDsub} and supersolution of~\eqref{eq1283ieNDsuper} in $B_{1-\rho} $, 
and recalling~\eqref{3.9bis} and~\eqref{x1x2intherightset},
we have that 
\begin{equation}\label{subineq}
-\eta\Delta u^\ep(x_1^\delta)+\I^{1,r}u^\ep(x_1^\delta)+\I^{2,r}u^\ep(x_1^\delta)+f_1\le0
\end{equation} 
and 
\begin{equation}\label{superineq}
-\eta\Delta u_\ep(x_2^\delta)+\I^{1,r}u_\ep(x_2^\delta)+\I^{2,r}u_\ep(x_2^\delta)
+f_2\ge0.\end{equation} 
Thus, by subtracting \eqref{superineq} to \eqref{subineq} and using \eqref{holdthmdis4} and \eqref{holdthmdisbis}, we obtain
\begin{equation}\label{subineq-superin}\I^{1,r}u^\ep(x_1^\delta)
-\I^{1,r}u_\ep(x_2^\delta)+f_1-f_2-C(\beta +\|u\|_{L^\infty(\R^N)})\le0.
\end{equation}
Next, let us estimate the term
$\I^{1,r}u^{\ep}(x_1^\delta)-\I^{1,r}u_{\ep}(x_2^\delta)$
and show that it contains a main negative part. For $0<\nu_0<1$,
let us denote by $A_r$ the cone
$$A_r:=\left\{z\in B_r\,,\,|z\cdot (x_1^\delta-x_2^\delta)|\geq
\nu_0|z||x_1^\delta-x_2^\delta|\right\}.$$ Then {\footnotesize{
\begin{equation}\label{I1r}\begin{split}
&\I^{1,r}u^{\ep}(x_1^\delta)-\I^{1,r}u_{\ep}(x_2^\delta)\\=\;&
-\int_{A_r}\Big[u^{\ep}(x_1^\delta+z)-u^{\ep}(x_1^\delta)
-\nabla u^{\ep}(x_1^\delta)\cdot z -(
u_{\ep}
(x_2^\delta+z)-u_{\ep}(x_2^\delta)-\nabla u_{\ep}(x_2^\delta)\cdot
z )\Big]K(z)\,dz\\& -
\int_{B_r\setminus
A_r}\Big[u^{\ep}(x_1^\delta+z)-u^{\ep}(x_1^\delta)
-\nabla u^{\ep}(x_1^\delta)\cdot z -(
u_{\ep}
(x_2^\delta+z)-u_{\ep}(x_2^\delta)-\nabla u_{\ep}(x_2^\delta)\cdot
z )\Big]K(z)\,dz
\\=:\;&-T_1-T_2.\end{split}\end{equation}}}
\noindent {F}rom \eqref{holdthmdis}
we have 
\begin{equation}\label{T2}T_2\leq C(\beta +\|u\|_{L^\infty(\R^N)}).\end{equation}
Let us
estimate $T_1$. Using \eqref{holdthmdis2} and \eqref{holdthmdis3}, and successively making the change of variable $z\rightarrow -z$, we get the following estimate of $T_1$: {\footnotesize{
 \begin{equation*}\begin{split} T_1&\leq
\int_{A_r}\Big[L|x_1^\delta+z-x_2^\delta|^\al-L|x_1^\delta-x_2^\delta|^\al-\al
L|x_1^\delta-x_2^\delta|^{\al-2}(x_1^\delta-x_2^\delta)\cdot
z\Big]K(z)\,dz\\&\qquad
+C(\beta +\|u\|_{L^\infty(\R^N)})\\&\qquad +\int_{A_r}\Big[L|x_1^\delta-z-x_2^\delta|^\al
-L|x_1^\delta-x_2^\delta|^\al+\al
L|x_1^\delta-x_2^\delta|^{\al-2}(x_1^\delta-x_2^\delta)\cdot
z\Big]K(z)\,dz\\&=2\int_{A_r}\Big[L|x_1^\delta+z-x_2^\delta|^\al-L|x_1^\delta-x_2^\delta|^\al-\al
L|x_1^\delta-x_2^\delta|^{\al-2}(x_1^\delta-x_2^\delta)\cdot
z\Big]K(z)\,dz\\&\qquad+C(\beta +\|u\|_{L^\infty(\R^N)})
\\& \leq\al L\int_{A_r}\sup_{\{|t|\leq1\}}\Big\{|x_1^\delta-x_2^\delta+tz|^{\al-4}
\big(|x_1^\delta-x_2^\delta+tz|^2|z|^2
-(2-\al)[(x_1^\delta-x_2^\delta+tz)\cdot
z]^2\big)\Big\}K(z)\,dz\\&\qquad +C(\beta +\|u\|_{L^\infty(\R^N)}).\end{split}\end{equation*}}}
\noindent Let us fix
$r:=\sigma|x_1^\delta-x_2^\delta|$, for some~$\sigma>0$. Then, for $z\in A_r$,
\begin{eqnarray*}
&& |x_1^\delta-x_2^\delta+tz|\leq
(1+\sigma)|x_1^\delta-x_2^\delta|\\
{\mbox{and }} && |(x_1^\delta-x_2^\delta+tz)\cdot
z|\geq|(x_1^\delta-x_2^\delta)\cdot z|-|z|^2\geq
\left(\nu_0-\sigma\right)|x_1^\delta-x_2^\delta||z|.\end{eqnarray*}
Let us
choose $0<\sigma<\nu_0<1$ such  that
$$C_0:=-(1+\sigma)^2+(2-\al)(\nu_0-\sigma)^2>0,$$then by \eqref{KERNEL},
 \begin{equation}\label{T_1} \begin{split}T_1&\leq
-CC_0 L|x_1^\delta-x_2^\delta|^{\al-2}
\int_{A_r}|z|^2K(z)\,dz +C(\beta +\|u\|_{L^\infty(\R^N)})\\&\leq -CC_0 L|x_1^\delta 
-x_2^\delta|^{\al-2}r^{2-2s}+C(\beta +\|u\|_{L^\infty(\R^N)})\\&
\leq  -CC_0 L|x_1^\delta -x_2^\delta|^{\al-2s}+C(\beta +\|u\|_{L^\infty(\R^N)}).
\end{split}\end{equation}
From \eqref{subineq-superin}, \eqref{I1r}, \eqref{T2} and  \eqref{T_1}, we obtain
\begin{equation*}\begin{split}CC_0 L|x_1^\delta -x_2^\delta|^{\al-2s}-C(\beta +\|u\|_{L^\infty(\R^N)})+f_1-f_2\le 0.
\end{split}\end{equation*}
Letting $\delta$ go to~$0$, the last inequality and \eqref{eq::s12} yield
\begin{equation*}CC_0L|\xs_1-\xs_2|^{\al-2s}\leq C(\beta +\|u\|_{L^\infty(\R^N)})+f_2-f_1.\end{equation*}
Thus, since  $\al-2s<0$, using \eqref{eq::s11} and letting $\beta$ go to $0$,
we finally obtain
\begin{equation*}
L\le C\left(f_2-f_1+\|u\|_{L^\infty(\R^N)}\right)^\frac{\alpha}{2s}\|u\|_{L^\infty(\R^N)}^{1-\frac{\alpha}{2s}}.
\end{equation*}
Since $L$ was chosen as big as we wish, we get   a contradiction
for
$$L> C\left(f_2-f_1+\|u\|_{L^\infty(\R^N)}\right)^\frac{\alpha}{2s}\|u\|_{L^\infty(\R^N)}^{1-\frac{\alpha}{2s}}.$$ This proves \eqref{mainholdest}.
\end{proof}

With the aid of Lemma~\ref{etaregularitylemmaND}, we can prove
the following regularity result (with uniform bounds):

\begin{lemma}\label{IL NUOVO}
Let~$T>1$, $\eta$, ${\mu}\in (0,1)$, $\rho>0$,  $\zeta\in\mathcal{Z}$.
Let~$Q\in L^\infty(\R)$ be a solution 
of
$$-\eta \ddot{Q}(x)+{\mu (Q-Q^\sharp_{\zeta_1,\zeta_2})}+{\mathcal{L}}(Q)(x)+a(x)  W'(Q(x))=0,
\quad {\mbox{ for any }} x\in(-4T,4T).$$
Suppose that
\begin{equation}\label{097ihkGHJHG:00}
{\mbox{$Q(x)\in\overline{B_\rho(\zeta)}$ for any~$x\in
(-4T,4T)$. 
}}\end{equation}
Then, for any $\alpha\in(0,2s),$
\begin{equation}\label{holdercleaninterval}
[Q]_{C^{0,\alpha}(-T,T)}\le
C T^{-\alpha} \left(\|Q\|_{L^\infty(\R)}+
T^{2s}\rho+{T^{2s}\mu}\right)^{\frac{\alpha}{2s}}\rho^{1-\frac{\alpha}{2s}}
,\end{equation}
for some~$C>0$ independent of $\eta$ and depending on~structural constants.
\end{lemma}

\begin{proof}
Up to a translation, we assume that~$\zeta=0$,
hence~\eqref{097ihkGHJHG:00} becomes
\begin{equation}\label{097ihkGHJHG}
{\mbox{$|Q(x)|\le\rho$, for any~$x\in
(-4T,4T)$.
}}\end{equation}
We let~$\tau_o\in C^\infty_0([-4,4],[0,1])$ be
such that~$\tau_o(x)=1$ for any~$x\in \left[ -3,3\right]$.
We define~$\tau(x):= \tau_o(x/T)$ and~$u(x):=\tau(x)\, Q(x)$.
Notice that, by~\eqref{097ihkGHJHG},
\begin{equation}\label{9iokdf67896pp}
{\mbox{$|u(x)|\le\rho$ for any~$x\in
\R$.}}
\end{equation}
Arguing as in  Lemma 4.1 in~\cite{MR3594365}, we see that~$u$ is solution of 
$$-\eta \ddot{u}+\mathcal{L}(u)=f\quad\text{in }(-2T,2T),$$
for some function $f$ satisfying
\begin{equation*}
\|f\|_{L^\infty(-2T,2T)}\le 
\frac{C\|Q\|_{L^\infty(\R)}}{T^{2s}}+
C\rho+{C\mu},\end{equation*}
with~$C>0$ independent of~$\eta$. 
Let~$v(x):=u(Tx)$, then~$v$ is a solution of 
$$-\eta T^{2(s-1)} \ddot v+\mathcal{L}(v)=T^{2s}f\quad\text{in }(-2,2).$$
Therefore, by Lemma~\ref{etaregularitylemmaND}
and~\eqref{9iokdf67896pp}, we have that, for any~$\alpha\in(0, 2s)$,
$$[v]_{C^{0,\alpha}(-1,1)}\leq
C\left(\|Q\|_{L^\infty(\R)}+T^{2s}\rho+{T^{2s}\mu}\right)^{\frac{\alpha}{2s}}
\rho^{1-\frac{\alpha}{2s}}.$$
Scaling back we get \eqref{holdercleaninterval}.
\end{proof}

\section{Energy estimates}\label{S:3}

The goal of this section is to provide suitable integral estimates,
with the aim of bounding the energy from below (this bound is crucial
to apply minimization methods in the variational arguments
exploited in the forthcoming Lemma~\ref{78:77:76}).
More precisely, we provide a bound on the ``mixed term''
of the energy functional, as defined in~\eqref{0q0303030303s}.
We remark that this is somehow an ``unpleasant''
term to consider, since, differently from the classical case,
it cannot be reabsorbed into the quadratic terms in the energy
by the Cauchy-Schwarz inequality, since it would produce
infinite contributions when arguing in this way.

\begin{lemma}\label{LEMMA1} 
Let~$v\in L^\infty(\R)$.
Then
\begin{eqnarray*}&& 
\big|\B_{\R,\R}(v,Q_{\zeta_1,\zeta_2}^\sharp) \big| \le {
\const \|Q_{\zeta_1,\zeta_2}^\sharp\|_{C^1(\R)}\,\Big(
[v]_{K,\R\times\R}+\|v\|_{L^2(\R)}\Big)}.\end{eqnarray*}
\end{lemma}

\begin{proof} We set~$I_-:=(-\infty,-2)$, $I_+:=(2,+\infty)$ and~$J:=[-2,2]$, and we
notice that~$\B_{I_-,I_-}(v,Q_{\zeta_1,\zeta_2}^\sharp)=\B_{I_+,I_+}(v,Q_{\zeta_1,\zeta_2}^\sharp)=0$,
since~$Q_{\zeta_1,\zeta_2}^\sharp$ is constant on~$I_-\cup I_+$.
Using this and~\eqref{BIJJI},
we see that
\begin{equation}\label{STI:0}\begin{split}
\B_{\R,\R}(v,Q_{\zeta_1,\zeta_2}^\sharp) &=
\B_{J,J}(v,Q_{\zeta_1,\zeta_2}^\sharp)+
2\B_{J,I_-}(v,Q_{\zeta_1,\zeta_2}^\sharp)\\&
+2\B_{I_-,I_+}(v,Q_{\zeta_1,\zeta_2}^\sharp)
+2\B_{J,I_+}(v,Q_{\zeta_1,\zeta_2}^\sharp).
\end{split}\end{equation}
Moreover, if~$x\in[-2,1]$ and~$y\in(2,+\infty)$ we have that
$$ |x-y|=y-x\ge\frac{y}{2}+1-1=\frac{y}{2},$$
and hence, recalling~\eqref{KERNEL},  we have that
\begin{eqnarray*}&&
\iint_{[-2,1]\times (2,+\infty)} 
\big|Q_{\zeta_1,\zeta_2}^\sharp(x)-\zeta_2\big|^2\,K(x-y)\,dx\,dy\\&\le&
\const \|Q_{\zeta_1,\zeta_2}^\sharp\|_{L^\infty(\R)}^2
\iint_{[-2,1]\times (2,+\infty)} y^{-1-2s}\,dx\,dy\\&\le&
\const \|Q_{\zeta_1,\zeta_2}^\sharp\|_{L^\infty(\R)}^2 .
\end{eqnarray*} 
Therefore, by the Cauchy-Schwarz inequality we find that
\begin{equation}\label{STI:1}
\begin{split}&\big|
\B_{J,I_{+}}(v,Q_{\zeta_1,\zeta_2}^\sharp)\big|\\
=\,&\left|
\iint_{[-2,2]\times (2,+\infty)}\big( v(x)-v(y)\big)
\big(Q_{\zeta_1,\zeta_2}^\sharp(x)-\zeta_2\big)\,K(x-y)\,dx\,dy\right|
\\
\le\,&   \sqrt{
\iint_{[-2,1]\times (2,+\infty)} \big|
v(x)-v(y)\big|^2\,K(x-y)\,dx\,dy}\\
&\qquad\times\,\sqrt{
\iint_{[-2,1]\times (2,+\infty)} \big|
Q_{\zeta_1,\zeta_2}^\sharp(x)-\zeta_2\big|^2\,K_(x-y)\,dx\,dy}
\\ \le\,& \const \|Q_{\zeta_1,\zeta_2}^\sharp\|_{L^\infty(\R)}\,
[v]_{K,\R\times\R}
.\end{split}
\end{equation}
Similarly, we see that
\begin{equation}\label{STI:2}
\big|\B_{J,I_-}(v,Q_{\zeta_1,\zeta_2}^\sharp) \big|\le
\const \|Q_{\zeta_1,\zeta_2}^\sharp\|_{L^\infty(\R)}\,
[v]_{K,\R\times\R}.\end{equation}
Also, we have that
\begin{equation}\label{STI:33}\begin{split}
&\big|\B_{I_-,I_+}(v,Q_{\zeta_1,\zeta_2}^\sharp) \big|=
\left|
\iint_{(-\infty,-2)\times (2,+\infty)} 
\big( v(x)-v(y)\big)\big(\zeta_1-\zeta_2\big)\,K(x-y)\,dx\,dy\right|
\\ &\le
\const\,\|Q_{\zeta_1,\zeta_2}^\sharp\|_{L^\infty(\R)}
\iint_{(-\infty,-2)\times (2,+\infty)} \big(
|v(x)|+|v(y)|\big)\,(y-x)^{-1-2s}\,dx\,dy
\\ &\le
\const\,\|Q_{\zeta_1,\zeta_2}^\sharp\|_{L^\infty(\R)}\left[
\int_{(-\infty,-2)} \frac{|v(x)|}{(2-x)^{2s}}\,dx+\int_{(2,+\infty)} \frac{|v(y)|}{(y+2)^{2s}}\,dy
\right].
\end{split}
\end{equation}
In addition, using the Cauchy-Schwarz inequality we see that
\begin{equation}\label{00e:02eifrjgjeo38}
\begin{split}
\int_{(2,+\infty)} \frac{|v(y)|}{(y+2)^{2s}}\,dy\,&
\le \sqrt{ \int_{\R} |v(y)|^2\,dy\,
\int_{\R} \frac{dy}{(y+2)^{4s}}
}\\&\le  \const\,\sqrt{ \int_{\R} |v(x)|^2\,dx
}.
\end{split}
\end{equation}
We stress that we have used condition~\eqref{RANGE} here. 
Similarly,
$$ \int_{(-\infty,-2)} \frac{|v(x)|}{(2-x)^{2s}}\,dx\le\const\,\sqrt{ \int_{\R} |v(x)|^2\,dx}.$$
Plugging this and~\eqref{00e:02eifrjgjeo38} into~\eqref{STI:33},
we thus conclude that
\begin{equation}\label{ILaidnfg234} \big|\B_{I_-,I_+}(v,Q_{\zeta_1,\zeta_2}^\sharp) \big|\le\const\,\|Q_{\zeta_1,\zeta_2}^\sharp\|_{L^\infty(\R)}\,\sqrt{ \int_{\R} |v(x)|^2\,dx}.\end{equation}
Furthermore, by the Cauchy-Schwarz inequality and~\eqref{NORM}, we have that {\footnotesize{
\begin{equation}\label{STI:4}\begin{split}&
\big|\B_{J,J}(v,Q_{\zeta_1,\zeta_2}^\sharp) \big|\\ \le\;& 
\sqrt{\iint_{J\times J} \big| v(x)-v(y)\big|^2\,K(x-y)\,dx\,dy\;
\iint_{J\times J} \big|( Q_{\zeta_1,\zeta_2}^\sharp)(x)-Q_{\zeta_1,\zeta_2}^\sharp(y)\big|^2\,K(x-y)\,dx\,dy}
\\
\le\;& \const[v]_{K,\R\times\R}\sqrt{\iint_{J\times J} \big|( Q_{\zeta_1,\zeta_2}^\sharp)(x)-Q_{\zeta_1,\zeta_2}^\sharp(y)\big|^2\,K(x-y)\,dx\,dy}.
\end{split}\end{equation}}}
\noindent Now, using that~$Q_{\zeta_1,\zeta_2}^\sharp(x)=\zeta_1$
for any~$x\in\left(-\infty,-1\right)$ and~$Q_{\zeta_1,\zeta_2}^\sharp(x)=\zeta_2$
for any~$x\in\left(1,+\infty\right)$, we have that
\begin{equation*}\begin{split}
&\iint_{J\times J} \big|( Q_{\zeta_1,\zeta_2}^\sharp)(x)-Q_{\zeta_1,\zeta_2}^\sharp(y)\big|^2\,K(x-y)\,dx\,dy\\&\qquad
=\int_{-2}^{2}\int_{-2}^2 \big|( Q_{\zeta_1,\zeta_2}^\sharp)(x)-Q_{\zeta_1,\zeta_2}^\sharp(y)\big|^2\,K(x-y)\,dx\,dy
\le \const\|Q_{\zeta_1,\zeta_2}^\sharp\|_{C^1(\R)}^2.
\end{split}
\end{equation*}
{F}rom this and~\eqref{STI:4}, we find that
$$ \big|\B_{J,J}(v,Q_{\zeta_1,\zeta_2}^\sharp) \big|\le
\const[v]_{K,\R\times\R}\,\|Q_{\zeta_1,\zeta_2}^\sharp\|_{C^1(\R)}.$$
Combining this with~\eqref{STI:0}, \eqref{STI:1}, \eqref{STI:2}
and~\eqref{ILaidnfg234}, we obtain the desired result.
\end{proof}

\section{Variational methods and
constrained minimization for a perturbed problem}\label{S:4}

Fixed~$\zeta_1$, $\zeta_2\in{\mathcal{Z}}$
and~$r\in(0,\,\min\{\delta_0,r_0\}]$ (where~$r_0$ and~$\delta_0$
are given by \eqref{KERNEL} and~\eqref{POT:GROW}, respectively),
we construct constrained minimizers for our variational problems. To this aim,
we take~$b_1\le-1$ and~$b_2\ge1$ and consider~$\phi$ and~$\psi$ solutions to 
\begin{equation}\label{phi}
\begin{cases}
-\eta\ddot \phi+\mathcal{L}\phi=
C_0&\text{in }(b_1-\tau,b_2+\tau),\\
\phi=\zeta_1+r&\text{in }(-\infty, b_1-\tau],\\
\phi=\zeta_2+r&\text{in }[b_2+\tau,+\infty),
\end{cases}
\end{equation}
and 
\begin{equation}\label{psi}
\begin{cases}
-\eta\ddot\psi+\mathcal{L}\psi=
-C_0&\text{in }(b_1-\tau,b_2+\tau),\\
\psi=\zeta_1-r&\text{in }(-\infty, b_1-\tau],\\
\psi=\zeta_2-r&\text{in }[b_2+\tau,+\infty),
\end{cases}
\end{equation}
where
\begin{equation}\label{decizer}
C_0:=\|aW'\|_{L^\infty(\R)}+2|\zeta_1|+2|\zeta_2|+1\end{equation}  and~$\tau\in(0,1)$. 
It is known that solutions to~\eqref{phi} and~\eqref{psi}
with~$\eta=0$ grow like~$d^s(x)$ plus the boundary data
away from the boundary of~$(b_1-\tau,b_2+\tau)$,
where~$d(x)$ is the distance function to the boundary 
of~$(b_1-\tau,b_2+\tau)$,  see~\cite{rosotonserra}. 
Thus, by stability of viscosity solutions, there exist~$c$, $C>0$
such that, for~$\tau$ small enough, {\footnotesize{
$$\begin{cases}
c (x-b_1+\tau)^s+o_\eta(1)\le \phi(x)-\zeta_1-r\leq
C (x-b_1+\tau)^s+o_\eta(1)&\text{for }x\in[b_1-\tau,b_1],\\
c (b_2+\tau-x)^s+o_\eta(1)\le \phi(x)-\zeta_2-r\leq
C (b_2+\tau-x)^s+o_\eta(1)&\text{for }x\in[b_2,b_2+\tau],\\
-C(x-b_1+\tau)^s+o_\eta(1)\le \psi(x)-\zeta_1+r\leq
-c(x-b_1+\tau)^s+o_\eta(1)&\text{for }x\in[b_1-\tau,b_1],\\
-C(b_2+\tau-x)^s+o_\eta(1)\le \psi(x)-\zeta_2+r\leq -c (b_2+\tau-x)^s+o_\eta(1)&\text{for }x\in[b_2,b_2+\tau],
\end{cases}
$$}}
\noindent where $o_\eta(1)\to 0$ as~$\eta\searrow0$. In particular,
for small~$\tau$,
\begin{equation}\label{etaC}
\begin{cases}
|\phi(x)-\zeta_1-r|\leq \displaystyle\frac r4&\text{for }
x\in[b_1-\tau,b_1],\\
\\
|\phi(x)-\zeta_2-r|\leq \displaystyle
\frac r4&\text{for }x\in[b_2,b_2+\tau],\\
\\
| \psi(x)-\zeta_1+r|\le \displaystyle
\frac r4&\text{for }x\in[b_1-\tau,b_1],\\
\\
|\psi(x)-\zeta_2+r|\le\displaystyle\frac r4&\text{for }x\in[b_2,b_2+\tau].
\end{cases}
\end{equation}
Next, consider  smooth functions~$\Phi:\R\to\R$ and~$\Psi:\R\to\R$  such that
\begin{equation}\label{Phidef}\begin{cases}
\Phi(x)=\phi(x)&\text{for }x\in (-\infty,b_1-2\tau]
\cup[b_2+2\tau,+\infty),\\
\zeta_1+\displaystyle\frac34r\leq\Phi(x)\leq \phi(x)\leq
\zeta_1+\displaystyle\frac54 r&\text{for }x\in(b_1-2\tau,b_1],\\
\Phi(x)\geq \phi(x)&\text{for }x\in(b_1,b_2),\\
\zeta_2+\displaystyle
\frac34r\leq\Phi(x)\leq \phi(x)\le \zeta_2+ \displaystyle
\frac54 r &\text{for }[b_2,b_2+2\tau)
\end{cases}
\end{equation}
and 
\begin{equation}\label{Psidef}\begin{cases}
\Psi(x)=\psi(x)&\text{for all }x\in (-\infty,b_1-2\tau]\cup[b_2+2\tau,+\infty),\\
 \zeta_1-\displaystyle
\frac54 r\le\psi(x)\le \Psi(x)\leq\zeta_1-\displaystyle
\frac34r&\text{for all }x\in(b_1-2\tau,b_1],\\
 \Psi(x)\leq \psi(x)&\text{for all }x\in(b_1,b_2),\\
\zeta_2-\displaystyle
\frac54r \le \psi(x)\le \Psi(x)\leq\zeta_2-\displaystyle
\frac34r&\text{for all }[b_2,b_2+2\tau).
\end{cases}
\end{equation}
With this, we can define the set
\begin{equation}\begin{split}
\label{Gamma} \Gamma(b_1,b_2):=\;&
\Big\{ Q:\R\to\R {\mbox{ s.t. }}
Q-Q_{\zeta_1,\zeta_2}^\sharp\in H^1(\R), \\&
\Psi(x)\le Q(x)\le \Phi(x) {\mbox{ for all }}
x\in (-\infty,b_1]\cup[b_2,+\infty)
\Big\}.\end{split}\end{equation}
Given~$\eta$, ${\mu}\in(0,1]$, we also consider the energy functional
\begin{equation} \label{01edkIETA}\begin{split}
I_{\eta,{\mu}}(Q)\,&:=
\frac\eta2 \int_{\R} |\dot Q(x)|^2\,dx+{\frac\mu2\int_\R
\big|Q(x)-Q_{\zeta_1,\zeta_2}^\sharp(x)\big|^2\,dx}
+\int_\R a(x)W\big(Q(x)\big)\,dx
\\ &+\frac14
\iint_{\R\times \R} \Big( \big| Q(x)-Q(y)\big|^2
-\big| Q_{\zeta_1,\zeta_2}^\sharp(x)-Q_{\zeta_1,\zeta_2}^\sharp(y)\big|^2\Big)
\,K(x-y)\,dx\,dy.\end{split}
\end{equation}
Then, we can construct a constrained minimizer for~$I_{\eta,{\mu}}$ in~$\Gamma(b_1,b_2)$
(later on, in Proposition~\ref{FREE-eta}, we will take~$b_1$ and~$b_2$ conveniently
separated, in order to employ condition~\eqref{a nondegenerate},
so
to obtain an unconstrained minimizer, and then, in 
Section~\ref{S023fhYU0193933}, we will send~$\eta\searrow0$
for a fixed~$\mu>0$.
Finally, in Section~\ref{UPAlao}, we will send~$\mu\searrow0$,
in order to obtain a solution of our original equation, as claimed in Theorem~\ref{MAIN}).

\begin{lemma}\label{78:77:76}
There exists~$Q_{\eta,{\mu}}\in\Gamma(b_1,b_2)$ such that
\begin{equation}
\label{LE:X1} I_{\eta,{\mu}}(Q_{\eta,{\mu}})
\le I_{\eta,{\mu}}(Q) {\mbox{ for all }} Q\in\Gamma(b_1,b_2).\end{equation}
Also, setting~$v_{\eta,{\mu}}:=Q_{\eta,{\mu}}-Q_{\zeta_1,\zeta_2}^\sharp$, we have that
\begin{eqnarray}
\label{LE:X2} && [v_{\eta,{\mu}}]_{H^1(\R)}\le\frac{\kappa}{\sqrt{\eta\mu}},\\
\label{LE:X3} && [v_{\eta,{\mu}}]_{K,\R\times\R}\le\frac{\kappa}{\sqrt{\mu}},\\
\label{LE:X4} && \|v_{\eta,{\mu}}\|_{L^\infty(\R)}\le{\kappa},\\
\label{LE:X5} &&
\|v_{\eta,{\mu}}\|_{L^2(\R)}\le\frac\kappa{\mu},\\
\label{PURT}&& [v_{\eta,{\mu}}]_{C^{0,\frac12}(\R)}\le\frac{\kappa}{\sqrt{\eta\mu}},\\
\label{DAFS823435rt10}&&
\min\{\zeta_1,\zeta_2\}\le Q_{\eta,\mu}(x)\le\max\{\zeta_1,\zeta_2\}\qquad
{\mbox{ for all }}x\in\R\\
\label{LE:X6} {\mbox{and }} &&
E_{\R^2}(Q_{\eta,{\mu}})\geq -\frac\kappa{\mu^2},
\end{eqnarray}
for some~$\kappa>0$, which
possibly depends on~$Q_{\zeta_1,\zeta_2}^\sharp$
and on structural constants.
\end{lemma}

\begin{proof} We take a minimizing sequence~$Q_j\in\Gamma(b_1,b_2)$
for the functional~$I_{\eta,{\mu}}$, and we let~$v_j:=Q_j-Q_{\zeta_1,\zeta_2}^\sharp\in H^1(\R)$.
Defining
$$Q^\star_j(x):=\begin{cases}
Q_j(x) & {\mbox{ if }} Q_j(x)\in \big(\min\{\zeta_1,\zeta_2\},\max\{\zeta_1,\zeta_2\}\big),\\
\min\{\zeta_1,\zeta_2\} & {\mbox{ if }} Q_j(x)\le\min\{\zeta_1,\zeta_2\},\\
\max\{\zeta_1,\zeta_2\} & {\mbox{ if }} Q_j(x)\ge\max\{\zeta_1,\zeta_2\},
\end{cases}$$
we see by direct inspection that
\begin{equation}\label{GasbdflergriAFgshdlflll12}
|Q^\star_j(x)-Q^\star_j(y)|\le
|Q_j(x)-Q_j(y)|\qquad{\mbox{and}}\qquad|\dot Q^\star_j(x)|\le |\dot Q_j(x)|.
\end{equation}
Also, we point out that~$W(Q^\star_j(x))=0\le W(Q(x))$ for every~$x\in\{Q^\star_j\ne Q_j\}$
and, as a result,
\begin{equation} \label{GasbdflergriAFgshdlflll13}
W(Q^\star_j(x))\leq W( Q_j(x)).
\end{equation}
Moreover, in light of~\eqref{TRAQQ}, if~$Q_j(x)\le\min\{\zeta_1,\zeta_2\}$ then
\begin{eqnarray*}
&& \big|Q^\star_j(x)-Q_{\zeta_1,\zeta_2}^\sharp(x)\big|=
\big|\min\{\zeta_1,\zeta_2\}-Q_{\zeta_1,\zeta_2}^\sharp(x)\big|
=Q_{\zeta_1,\zeta_2}^\sharp(x)-
\min\{\zeta_1,\zeta_2\}\\&&\qquad\le Q_{\zeta_1,\zeta_2}^\sharp(x)-Q_j(x)\le
\big|Q_j(x)-Q_{\zeta_1,\zeta_2}^\sharp(x)\big|,\end{eqnarray*}
and a similar estimate holds if~$Q_j(x)\ge\max\{\zeta_1,\zeta_2\}$.

This gives that~$\big|Q^\star_j(x)-Q_{\zeta_1,\zeta_2}^\sharp(x)\big|\le
\big|Q_j(x)-Q_{\zeta_1,\zeta_2}^\sharp(x)\big|$ for all~$x\in\R$. Consequently,
by~\eqref{GasbdflergriAFgshdlflll12} and~\eqref{GasbdflergriAFgshdlflll13},
we see that~$I_{\eta,\mu}(Q^\star_j)\le I_{\eta,\mu}(Q_j)$.

For that reason, from now on, possibly replacing~$Q_j$ with~$Q^\star_j$,
we can suppose that
\begin{equation}\label{DAFS823435rt11}
\min\{\zeta_1,\zeta_2\}\le Q_j(x)\le\max\{\zeta_1,\zeta_2\}\qquad
{\mbox{ for all }}x\in\R,
\end{equation}
and therefore
\begin{equation}\label{2397777x66x02kdkjroo3455o45}
|v_j(x)|\le\kappa\qquad
{\mbox{ for all }}x\in\R,
\end{equation}
for some~$\kappa>0$.

We also define~$J_{\eta,\mu}(v):=I_{\eta,\mu}(Q_{\zeta_1,\zeta_2}^\sharp+v)$.
In this way, the sequence~$v_j$ is minimizing for~$J_{\eta,\mu}$, and
\begin{equation}\label{0-102o3:2}\begin{split}
J_{\eta,\mu}(v)\,&=\frac\eta2 \int_{\R} \big(|\dot Q_{\zeta_1,\zeta_2}^\sharp(x)|^2
+|\dot v(x)|^2+2\dot Q_{\zeta_1,\zeta_2}^\sharp(x) \dot v(x)\big)
\,dx+\frac\mu2\int_{\R}|v(x)|^2\,dx\\&\qquad+\int_\R a(x)W\big(Q_{\zeta_1,\zeta_2}^\sharp(x)+v(x)\big)\,dx
\\ &\qquad+\frac14
\iint_{\R\times \R} \big| v(x)-v(y)\big|^2\,K(x-y)\,dx\,dy+\frac12\B_{\R,\R}(v,Q_{\zeta_1,\zeta_2}^\sharp).
\end{split}\end{equation}
Since~$v_j$ is minimizing and the zero function is an admissible
competitor for~$J_{\eta,\mu}$, we can also suppose that
\begin{equation}\label{0-102o3:3}
J_{\eta,\mu}(v_j)\le J_{\eta,\mu}(0)+1\le
\frac12 \int_{\R}|\dot Q_{\zeta_1,\zeta_2}^\sharp(x)|^2
\,dx+\int_\R\overline a\, W\big(Q_{\zeta_1,\zeta_2}^\sharp(x)\big)\,dx+1
\le\kappa.
\end{equation}
In addition, by Cauchy-Schwarz inequality,
$$ 2\,\big|
\dot Q_{\zeta_1,\zeta_2}^\sharp(x)\cdot \dot v_j(x)\big|\le
4\,\big|\dot Q_{\zeta_1,\zeta_2}^\sharp(x)\big|^2+
\frac14\,\big|\dot v_j(x)\big|^2.$$
Combining this estimate with formulas~\eqref{0-102o3:2} and~\eqref{0-102o3:3},
we conclude that
\begin{equation}\label{quasnd518}\begin{split}&
\frac{3\eta}8 \int_{\R} |\dot v_j(x)|^2\,dx+\frac\mu2\int_{\R}|v_j(x)|^2\,dx
+\frac14
\iint_{\R\times \R} \big| v_j(x)-v_j(y)\big|^2\,K(x-y)\,dx\,dy\\&\qquad\qquad+\frac12\B_{\R,\R}(v_j,Q_{\zeta_1,\zeta_2}^\sharp)
\le\kappa,\end{split}\end{equation}
up to the freedom of renaming~$\kappa$. 

Furthermore, recalling Lemma~\ref{LEMMA1}, fixing a small
additional parameter~$\e>0$, and using the Cauchy-Schwarz inequality,
\begin{equation}\label{0iASyye734}
\big|\B_{\R,\R}(v_j,Q_{\zeta_1,\zeta_2}^\sharp) \big| \le \kappa\,\Big(
[v_j]_{K,\R\times\R}+\|v_j\|_{L^2(\R)}\Big)\le \kappa\,\left(
\e\,[v_j]_{K,\R\times\R}^2+\e\,\|v_j\|^2_{L^2(\R)}+\frac1\e\right).
\end{equation}
Hence, in view of~\eqref{quasnd518} and choosing~$\e$ conveniently
small (possibly in dependence of~$\mu$),
\begin{equation}\label{093u9876543erer-23i0}
\frac{3\eta}8 [ v_j]_{H^1(\R)}^2\,dx+\frac\mu4 \|v_j\|^2_{L^2(\R)}
+\frac18[v_j]_{K,\R\times\R}^2
\le\frac\kappa\mu.\end{equation}
Accordingly,
we obtain that, up to a subsequence, $v_j$ converges locally uniformly in~$\R$ and weakly
in~$H^1(\R)$ and in the Hilbert space induced by~$[\cdot]_{K,\R\times \R}$ to a minimizer~$v_{\eta,\mu}$.
We then set~$Q_{\eta,\mu}:=v_{\eta,\mu}+Q_{\zeta_1,\zeta_2}^\sharp$ and we obtain~\eqref{LE:X1}.

Also, the claims in~\eqref{LE:X2}, \eqref{LE:X3} and~\eqref{LE:X5} 
follow by taking the limit in~\eqref{093u9876543erer-23i0}, as well as
the claim in~\eqref{LE:X4} follows by taking the limit in~\eqref{2397777x66x02kdkjroo3455o45}.

Moreover, the claim in~\eqref{PURT} follows from~\eqref{LE:X2}
and the one in~\eqref{DAFS823435rt10} is a consequence of~\eqref{DAFS823435rt11}.
Finally, to prove~\eqref{LE:X6} we observe that, in view of~\eqref{Eenergy}, \eqref{BIJJI}
and~\eqref{0iASyye734},
$$ E_{\R^2}(Q_j)=[v_j]^2_{K,\R\times\R}+2\B_{\R,\R}(v_j,Q_{\zeta_1,\zeta_2}^\sharp)\ge
\frac12[v_j]^2_{K,\R\times\R}-\frac12\|v_j\|^2_{L^2(\R)}-\kappa.
$$
Passing to the limit and making use of~\eqref{LE:X3}
and~\eqref{LE:X5} we obtain~\eqref{LE:X6}.
\end{proof}

Now we define
$$ J_*:= (b_{1},b_{2})$$
and
$$
L :=\big\{ x\in (-\infty,b_1]\cup [b_{2},+\infty) {\mbox{ s.t. }} \Psi(x)<Q_{\eta,{\mu}}(x)< \Phi(x)\big\}.
$$
Let also
\begin{equation}\label{Freeset}F:= J_*\cup L.\end{equation}
As usual, by taking inner variations, one sees that
in the set~$F$ the minimization problem is ``free''
and so it satisfies an Euler-Lagrange equation, as stated
explicitly in the next result:

\begin{lemma}\label{HJA:AA:2}
Let~$Q_{\eta,{\mu}}$ be as in Lemma~\ref{78:77:76}.
For any~$x\in F$, we have that
\begin{equation}\label{HJA:AA:2:EQ}
-\eta\,\ddot{Q}_{\eta,{\mu}}(x)+
{\mu\big(Q_{\eta,\mu}(x)-Q_{\zeta_1,\zeta_2}^\sharp(x)\big)}
+{\mathcal{L}}Q_{\eta,{\mu}}(x) + a(x)\, W'(Q_{\eta,{\mu}}(x)) =0.\end{equation} 
\end{lemma}

Now we define the set
\begin{equation}\label{Sigma}
\Sigma :=\Big\{ Q:\R\to\R {\mbox{ s.t. }}
Q-Q_{\zeta_1,\zeta_2}^\sharp\in H^1(\R) \, {\mbox{ and }}\,
\Psi(x)\le Q(x)\le \Phi(x) {\mbox{ for all }} x\in \R\Big\}.
\end{equation}
We notice that, differently from the
definition of~$\Gamma(b_1,b_2)$
given in~\eqref{Gamma}, we require here that a function~$Q$
belongs to~$\Sigma$ if it satisfies~$\Psi\le Q\le \Phi$ in the whole
of~$\R$, and not only in~$(-\infty,b_1]\cup [b_2,+\infty)$. 

As a matter of fact, we prove that
the minimizer~$Q_{\eta,{\mu}}\in\Gamma(b_1,b_2)$, given
by Lemma~\ref{78:77:76},  is actually a minimizer of~$I_{\eta,{\mu}}$
in~$\Sigma$:

\begin{lemma}\label{gammagammatilde}
Let~$Q_{\eta,{\mu}}$ be as in Lemma~\ref{78:77:76}. 
Then, we have that~$Q_{\eta,{\mu}}\in \Sigma$. In particular,  
\begin{equation}\label{dgeui:1}
\inf_{Q\in\Sigma} I_{\eta,{\mu}}(Q)=
\inf_{Q\in \Gamma(b_1,b_2)} I_{\eta,{\mu}}(Q)= I_{\eta,{\mu}}(Q_{\eta,{\mu}}).
\end{equation}
\end{lemma}

\begin{proof}
We first prove that~$Q_{\eta,{\mu}}$ belongs to~$\Sigma$. For this,
it is enough to show that
\begin{equation}\label{dgeui:0}
\Psi(x) \le Q_{\eta,{\mu}}(x)\le\Phi(x)
\quad {\mbox{for any }} x\in (b_1,b_2).
\end{equation}
To check this, we observe that, by Lemma~\ref{HJA:AA:2}, $Q_{\eta,{\mu}}$ is solution of 
$$-\eta\,\ddot{Q}_{\eta,{\mu}}(x)
+{\mu\big(Q_{\eta,\mu}(x)-Q_{\zeta_1,\zeta_2}^\sharp(x)\big)}
+{\mathcal{L}}Q_{\eta,{\mu}}(x) + a(x)\,
W'(Q_{\eta,{\mu}}(x)) =0$$
for any $x\in (b_1,b_2)$.

In addition, since $Q_{\eta,{\mu}}\in\Gamma(b_1,b_2)$, recalling~\eqref{Phidef} and~\eqref{Gamma}, 
we see that
\begin{equation}\label{dgeui}
Q_{\eta,{\mu}}(x)\leq \Phi(x)\leq\phi(x)\quad\text{for any }x
\in (-\infty, b_1]\cup [b_2,+\infty).\end{equation}  
We observe that
\begin{equation} \label{TPDAs12}
Q_{\eta,{\mu}}(x)\le \phi(x) \quad {\mbox{for any }} x\in (b_1,b_2).  
\end{equation}
To prove this, we define~$w:=Q_{\eta,{\mu}}-\phi$
and we suppose, by contradiction, in light of~\eqref{dgeui},
that~$w$ has a positive maximum
at some point~$x_\star\in(b_1,b_2)$. This gives that
$$ 0\le-\eta \ddot{w}(x_\star)+{\mathcal{L}} w(x_\star)=
-\eta \ddot{Q}_{\eta,{\mu}}(x_\star)+{\mathcal{L}} Q_{\eta,{\mu}}(x_\star)
+\eta \ddot{\phi}(x_\star)-{\mathcal{L}} \phi(x_\star).
$$
Hence, recalling~\eqref{phi} and~\eqref{HJA:AA:2:EQ},
and using also~\eqref{DAFS823435rt10},
\begin{eqnarray*} 0&\le&
-\mu\big(Q_{\eta,\mu}(x_\star)-Q_{\zeta_1,\zeta_2}^\sharp(x_\star)\big)
-a(x_\star)\, W'(Q_{\eta,\mu}(x_\star)) -C_0\\&\le&
2|\zeta_1|+2|\zeta_2|+\|aW'\|_{L^\infty(\R)}-C_0.\end{eqnarray*}
This is in contradiction with~\eqref{decizer}
and hence it completes the proof of~\eqref{TPDAs12}.

Consequently, in view of~\eqref{TPDAs12}, and
making again use of~\eqref{Phidef},
$$ Q_{\eta,\mu}(x)\le \phi(x) \le \Phi(x) \quad {\mbox{for any }}
x\in (b_1,b_2),$$
which proves the second inequality in~\eqref{dgeui:0}.
Similarly, one can check that 
$$ Q_{\eta,\mu}(x)\geq\Psi(x) \quad\text{for any }x\in(b_1,b_2),$$
which completes the proof of~\eqref{dgeui:0}.

Now, since~$Q_{\eta,{\mu}}\in\Sigma\subset \Gamma(b_1,b_2)$,
we have that
$$ \inf_{Q\in\Sigma} I_{\eta,{\mu}}(Q)\geq\inf_{Q\in
\Gamma(b_1,b_2)} I_{\eta,{\mu}}(Q)= I_{\eta,{\mu}}(Q_{\eta,{\mu}})\ge \inf_{Q\in\Sigma} I_{\eta,{\mu}}(Q),$$
which proves~\eqref{dgeui:1}. 
The proof of Lemma~\ref{gammagammatilde} is thus
complete.
\end{proof}

\section{Lewy-Stampacchia estimates and continuity results
for a double obstacle problem}\label{sec:LW}

In this section, we prove that constrained
minimizers of the perturbed problem
are continuous, with uniform bounds.
This estimate is based on a double obstacle
problem approach.
We follow a technique introduced by Lewy and Stampacchia
in~\cite{MR0271383} and suitably
modified in~\cite{MR3090147} to deal with nonlocal problems.
In our situation, differently from the previous literature,
we need to take into account the fact that the problem
is constrained by two obstacles. Moreover,
our problem is a superposition of a local and a nonlocal
operators and we aim at estimates which are uniform
with respect to the local contribution. The result that suits
for our purposes is the following:

\begin{proposition}\label{LEVSTA}
Let~$I$ be a bounded interval and~$f\in L^\infty(I)$.
Let~$u\in\Sigma$, with $\Sigma$ defined as in~\eqref{Sigma},
and assume that
\begin{equation}\label{89eiuwfhj2ery8375209138148}
\begin{split}&
\eta\int_\R \dot u(x)\, \big(\dot u(x)-\dot v(x)\big)\,dx\\&\qquad+\frac12\,
\iint_{\R^2} \big( u(x)-u(y)\big)\big( (u-v)(x)-(u-v)(y)\big)\,K(x-y)\,dx\,dy\\&\qquad
\le \int_\R f(x)\,(u-v)(x)\,dx,
\end{split}
\end{equation}
for every~$v\in\Sigma$ with~$v=u$ in~$\R\setminus I$.
Then,
\begin{equation}\label{LW}\begin{split}
\min\left\{\inf_{x\in I}-
|\ddot \Phi(x)| +{\mathcal{L}}\Phi(x),\,\inf_{x\in I} f(x)
\right\}&
\le
-\eta\ddot u(x) +{\mathcal{L}}u(x)\\&\le
\max\left\{\sup_{x\in I}
|\ddot \Psi(x)| +{\mathcal{L}}\Psi(x),\,\sup_{x\in I} f(x)
\right\}\end{split}
\end{equation}
in the sense of distributions.
\end{proposition}

\begin{proof} Let
\begin{equation}\label{MSTAR}
M^*:=\max\left\{\sup_{x\in I}
|\ddot \Psi(x)| +{\mathcal{L}}\Psi(x),\,\sup_{x\in I} f(x)
\right\}
\end{equation}
and
$$ I^*(v):=
\frac\eta2\,\int_I |\dot v(x)|^2\,dx+\frac14\,
\iint_{Q_I} \big| v(x)-v(y)\big|^2\,K(x-y)\,dx\,dy
- M^*\int_I v(x)\,dx,$$
where~$Q_I:=(I\times I)\cup\big( I\times(\R\setminus I)\big)
\cup\big( (\R\setminus I)\times I\big)$.
We take~$z^*$ to be a minimizer of~$I^*$ in the class
of functions~$v:\R\to\R$ with~$v(x)\le u(x)$ for any~$x\in\R$
and~$v(x)=u(x)$ for any~$x\in\R\setminus I$.
The existence of such minimizer follows by compactness,
along the lines given in the proof of Lemma~\ref{78:77:76}.
In particular,
\begin{equation}\label{z sotto}
{\mbox{$z^*(x)\le u(x) $
for any~$x\in\R$
and~$z^*(x)=u(x)$ for any~$x\in\R\setminus I$}}.
\end{equation}
Moreover, for any~$\e\in[0,1]$ 
and any~$w:\R\to\R$ with~$w(x)\le u(x)$ for any~$x\in\R$
and~$w(x)=u(x)$ for any~$x\in\R\setminus I$,
we have that~$z_\e(x):=\e w(x)+(1-\e)z^*(x)$
is an admissible competitor for~$z^*$
and consequently~$I^*(z_\e)\ge I^*(z^*)$, which gives that
\begin{equation}\label{w test}
\begin{split}
0\,\le\;&\frac{d}{d\e}I^*(z_\e)\Big|_{\e=0}\\=\;&
\eta\int_I \dot z^*(x) \big( \dot w(x)-\dot z^*(x)\big)\,dx\\&\qquad\qquad
+\frac12\,
\iint_{Q_I} \big( z^*(x)-z^*(y)\big)\big( (w-z^*)(x)-(w-z^*)(y)\big)
\,K(x-y)\,dx\,dy\\&\qquad\qquad
- M^*\int_I (w-z^*)(x)\,dx.
\end{split}
\end{equation}
We claim that
\begin{equation}\label{ILCLAIM0}
z^*\in \Sigma.
\end{equation}
To check this, we first use~\eqref{z sotto} to observe that
\begin{equation}\label{ILCLAIM1}
z^*(x)\le u(x)\le \Phi(x).
\end{equation}
Then, we take
\begin{equation*}
w^*(x):=
\max\{z^*(x),\Psi(x) \}=z^*(x)+\big( \Psi(x)-z^*(x)\big)_+\,.
\end{equation*}
By~\eqref{z sotto}, we know that~$w^*(x)\le u(x)$ for any~$x\in\R$.
Also, if~$x\in\R\setminus I$, we have that~$w^*(x)=
\max\{u(x),\Psi(x) \}=u(x)$. Therefore, we can make use of~\eqref{w test}
with~$w:=w^*$, and so we find that
\begin{equation}\label{PEZ:1}
\begin{split}
0\;&\le
\eta\int_{I\cap\{\Psi>z^*\}} \dot z^*(x) 
\big( \dot \Psi(x)-\dot z^*(x)\big)\,dx
\\ & +\frac12\,
\iint_{Q_I} \big( z^*(x)-z^*(y)\big)\Big( 
\big( \Psi(x)-z^*(x)\big)_+
-\big( \Psi(y)-z^*(y)\big)_+\Big)
\,K(x-y)\,dx\,dy\\&
- M^*\int_I \big( \Psi(x)-z^*(x)\big)_+\,dx.
\end{split}
\end{equation}
Furthermore, on~$\partial I$ we have that~$z^*=u\ge\Psi$,
hence,
from~\eqref{MSTAR} and integrating by parts, we see that
\begin{equation}\label{PEZ:2}
\begin{split}
&\eta\int_{I\cap\{\Psi>z^*\}} \dot \Psi(x) 
\big( \dot \Psi(x)-\dot z^*(x)\big)\,dx
\\ &+\frac12\,
\iint_{Q_I} \big( \Psi(x)-\Psi(y)\big)\big( 
\big( \Psi(x)-z^*(x)\big)_+
-\big( \Psi(y)-z^*(y)\big)_+\big)
\,K(x-y)\,dx\,dy\\&
- M^*\int_I \big( \Psi(x)-z^*(x)\big)_+\,dx\\=\;
&-\eta\int_{\R} \ddot \Psi(x) 
\big( \Psi(x)- z^*(x)\big)_+\,dx
\\ &+
\iint_{\R^2} \big( \Psi(x)-\Psi(y)\big)
\big( \Psi(x)-z^*(x)\big)_+
\,K(x-y)\,dx\,dy\\&
- M^*\int_\R \big( \Psi(x)-z^*(x)\big)_+\,dx
\\=\;
&\int_{\R} \Big(-\eta\ddot \Psi(x) +{\mathcal{L}}\Psi(x)
- M^*\Big) \big( \Psi(x)-z^*(x)\big)_+\,dx\\
\le\;&0
.
\end{split}
\end{equation}
Thus, subtracting \eqref{PEZ:1}
to~\eqref{PEZ:2}, we conclude that
\begin{equation}\label{PEZ:3}
\begin{split}
0\;&\ge
\eta\int_{I} \big( \dot \Psi(x)-\dot z^*(x)\big)
\big( \dot \Psi(x)-\dot z^*(x)\big)_+\,dx
\\ &\qquad+\frac12\,
\iint_{Q_I} \Big( \big( \Psi(x)-z^*(x)\big)-\big( \Psi(y)-z^*(y)\big)\Big)\\&\times\Big( 
\big( \Psi(x)-z^*(x)\big)_+
-\big( \Psi(y)-z^*(y)\big)_+\Big)
\,K(x-y)\,dx\,dy.
\end{split}
\end{equation}
The last term in~\eqref{PEZ:3} is nonnegative (see e.g. page~1115
in~\cite{MR3090147}), therefore
we get that
\begin{equation*}
0\ge
\int_{I} 
\big( \dot \Psi(x)-\dot z^*(x)\big)_+^2\,dx.
\end{equation*}
This says that~$\Psi(x)\le z^*(x)$ for any~$x\in I$ (and so for any~$x\in\R$,
due to~\eqref{z sotto}).
This and~\eqref{ILCLAIM1} imply~\eqref{ILCLAIM0}, as desired.

Then, from~\eqref{ILCLAIM0} we deduce that
both the minimum and the maximum between~$u$ and~$z^*$ belong
to~$\Sigma$, that is
\begin{eqnarray*}&&
v^\sharp(x):=\min\{u(x),z^*(x)\}=u(x)-
\big( u(x)-z^*(x)\big)_+\in \Sigma\\
{\mbox{and }}&& w^\sharp(x):=
\max\{u(x),z^*(x) \}=z^*(x)+\big( u(x)-z^*(x)\big)_+\in \Sigma\,.
\end{eqnarray*}
In particular, we can take~$v:=v^\sharp$ in~\eqref{89eiuwfhj2ery8375209138148}
and~$w:=w^\sharp$ in~\eqref{w test}. This gives that
\begin{equation}\label{8claO:1}
\begin{split}&
\eta\int_\R \dot u(x)\, \big( \dot u(x)-\dot z^*(x)\big)_+\,dx\\
&\qquad+\frac12\,
\iint_{\R^2} \big( u(x)-u(y)\big)\Big( 
\big( u(x)-z^*(x)\big)_+
-\big( u(y)-z^*(y)\big)_+\Big)\,K(x-y)\,dx\,dy\\&\qquad 
\le \int_\R f(x)\,\big( u(x)-z^*(x)\big)_+\,dx
\end{split}
\end{equation}
and
\begin{equation}\label{8claO:2}
\begin{split}&
M^*\int_I \big( u(x)-z^*(x)\big)_+\,dx\le
\eta\int_I \dot z^*(x) \big( \dot u(x)-\dot z^*(x)\big)_+\,dx\\&
+\frac12\,
\iint_{Q_I} \big( z^*(x)-z^*(y)\big)\Big( 
\big( u(x)-z^*(x)\big)_+-\big( u(y)-z^*(y)\big)_+\Big)
\,K(x-y)\,dx\,dy.
\end{split}
\end{equation}
Hence, subtracting~\eqref{8claO:2} to~\eqref{8claO:1}
and recalling~\eqref{MSTAR}, we obtain
\begin{equation*}\begin{split}&0\ge
\eta\int_I \big( \dot u(x)-\dot z^*(x)\big)
\big( \dot u(x)-\dot z^*(x)\big)_+\,dx\\&\qquad
+\frac12\,
\iint_{Q_I} \Big(
\big( u(x)-z^*(x)\big)-\big( u(y)-z^*(y)\big)\Big)\\&\times
\Big( 
\big( u(x)-z^*(x)\big)_+-\big( u(y)-z^*(y)\big)_+\Big)
\,K(x-y)\,dx\,dy.\end{split}
\end{equation*}
As above, this implies that~$u\le z^*$.
Combining this with~\eqref{z sotto}, we obtain
that~$z^*$ coincides with~$u$. As a consequence,
taking any function~$v\ge0$, supported in~$I$, and
defining~$w:=u-v$ in~\eqref{w test},
\begin{equation*}
\eta\int_I \dot u(x) \dot v(x)\,dx
+\frac12\,
\iint_{Q_I} \big( u(x)-u(y)\big)\big( v(x)-v(y)\big)
\,K(x-y)\,dx\,dy
\le M^*\int_I v(x)\,dx.
\end{equation*}
Integrating by parts the latter inequality, we obtain that
\begin{equation*}
\int_\R \Big(-\eta\ddot u(x) +{\mathcal{L}}u(x)\Big)\,v(x)\,dx
\le M^*\int_\R v(x)\,dx.
\end{equation*}
By duality, we thus obtain that
$$ -\eta\ddot u(x) +{\mathcal{L}}u(x)\le M^*,$$
in the sense of distributions, which is one of the inequalities in~\eqref{LW}.
The other inequality in~\eqref{LW} follows
by similar arguments.
\end{proof}

Using Lemma~\ref{etaregularitylemmaND},
Proposition~\ref{LEVSTA}
and a convolution method (see e.g. formula~(3.2) in~\cite{MR3161511}),
we obtain a useful uniform continuity
result for a perturbed problem. The statement that we
need for our purposes is the following:

\begin{corollary}\label{QCalphaobstaclecor}
Let~$Q_{\eta,{\mu}}$ be as in Lemma~\ref{78:77:76} and~$\alpha\in(0,2s)$. Then~$Q_{\eta,{\mu}}
\in C^{0,\alpha}(\R)$ and
\begin{equation}\label{ALSO5}
\|Q_{\eta,{\mu}}\|_{C^{0,\alpha}(\R)}\le \kappa,\end{equation}
for some~$\kappa> 0$, which possibly depends
on~$Q^\sharp_{\zeta_1,\zeta_2}$
and on structural constants.
\end{corollary}

\begin{proof} We take~$
v_{\eta,{\mu}}:=Q_{\eta,{\mu}}-Q_{\zeta_1,\zeta_2}^\sharp$,
as in Lemma~\ref{78:77:76}. By Lemma~\ref{gammagammatilde},
we know that~$Q_{\eta,{\mu}}\in\Sigma$.
We fix an interval~$I\subset\R$
and take any~$\xi\in \Sigma$.
For any~$\e\in(0,1)$, let~$\xi_\e:=\e\xi+(1-\e)Q_{\eta,{\mu}}=Q_{\eta,{\mu}}+\e(\xi-Q_{\eta,{\mu}})$.
Then~$\xi_\e\in \Sigma$ and therefore,
by~\eqref{dgeui:1}, we know that
\begin{eqnarray*}
0&\le& I_{\eta,{\mu}}(\xi_\e)-I_{\eta,{\mu}}(Q_{\eta,{\mu}})\\
&=& \frac\eta2\,\int_I \big(|\dot\xi_\e(x)|^2-|\dot Q_{\eta,{\mu}}(x)|^2\big)\,dx\\
&&
+ \frac\mu2\int_I\Big(
\big|\xi_\e(x)-Q_{\zeta_1,\zeta_2}^\sharp(x)\big|^2-
\big|Q_{\eta,{\mu}}(x)-Q_{\zeta_1,\zeta_2}^\sharp(x)\big|^2\Big)\,dx\\
&&
+\int_I a(x)\,\Big( W\big(\xi_\e(x)\big)-W\big(Q_{\eta,{\mu}}(x)\big)\Big)\,dx\\
&&+\frac14\,\iint_{\R\times\R}\Big(
\big| \xi_\e(x)-\xi_\e(y)\big|^2-\big| Q_{\eta,{\mu}}(x)-Q_{\eta,{\mu}}(y)\big|^2
\Big)\,K(x-y)\,dx\,dy\\
&=&
\e\eta\,\int_I \dot Q_{\eta,{\mu}}(x)\cdot\big(\dot\xi(x)-\dot Q_{\eta,{\mu}}(x)\big)\,dx\\
&&
+\e\mu\int_I\big(
Q_{\eta,\mu}(x)-Q_{\zeta_1,\zeta_2}^\sharp(x)\big)\cdot
\big( \xi(x)-Q_{\eta,\mu}(x)
\big)\,dx
\\&&+\e\int_I a(x)\,W'\big(Q_{\eta,{\mu}}(x)\big)\,\big(\xi(x)- Q_{\eta,{\mu}}(x)\big)dx\\
&&+\frac\e2\,\iint_{\R\times\R}\Big(
\big( Q_{\eta,{\mu}}(x)-Q_{\eta,{\mu}}(y)\big)\,\big( (\xi-Q_{\eta,{\mu}})(x)-(\xi-Q_{\eta,{\mu}})(y)\big)
\Big)\,K(x-y)\,dx\,dy\\
&&+o(\e).
\end{eqnarray*}
Thus, dividing this inequality by~$\e$ and sending~$\e\searrow0$,
we conclude that~$Q_{\eta,{\mu}}$ satisfies~\eqref{89eiuwfhj2ery8375209138148}
with
$$f:=-aW'(Q_{\eta,\mu})-\mu
(Q_{\eta,\mu} -Q_{\zeta_1,\zeta_2}^\sharp) .$$ Accordingly, by 
formula~\eqref{LW} in
Proposition~\ref{LEVSTA},
we know that
\begin{equation*}
-\const\le-\eta \ddot Q_{\eta,\mu}+{\mathcal{L}}Q_{\eta,\mu}\le\const
\end{equation*}
and therefore
\begin{equation}\label{ALSOIN}
-\kappa\le-\eta \ddot v_{\eta,\mu}+{\mathcal{L}}v_{\eta,\mu}\le\kappa
\end{equation}
in the sense of distributions,
for some~$\kappa> 0$, which possibly depends
on~$Q^\sharp_{\zeta_1,\zeta_2}$.

Now we take an even function~$\vartheta\in C^\infty_0([-1,1])$
and for any~$\e\in(0,1)$ we set~$\vartheta_\e(x):=\e^{-1}\vartheta(x/\e)$.
We consider the mollification~$v_{\eta,\mu,\e}:=v_{\eta,\mu} *\vartheta_\e$.
Notice that, as~$\e\searrow0$, we have that
\begin{equation}\label{ALSO4}
{\mbox{$v_{\eta,\mu,\e}$ converges locally uniformly to $v_{\eta,\mu}$,}}\end{equation}
thanks to~\eqref{PURT}.
Moreover, we observe that, for any~$\varphi\in C^\infty_0(\R)$,
\begin{eqnarray*}
&&\Big|\big( v_{\eta,\mu}(x)-v_{\eta,\mu}(y)\big)
\big( \varphi(x)-\varphi(y)\big)\,\vartheta_\e(z)\,K(x-y)\Big|\\
&&\qquad\le
\Big(
\big| v_{\eta,\mu}(x)-v_{\eta,\mu}(y)\big|^2\,K(x-y)
+
\big| \varphi(x)-\varphi(y)\big|^2\,K(x-y)\Big)\,\chi_{[-1,1]}(z)
,\end{eqnarray*}
which, as a function of~$(x,y,z)\in\R\times\R\times\R$, belongs
to~$L^1(\R\times\R\times\R)$, thanks to~\eqref{LE:X3}.
This implies that we can exploit the Dominated Convergence Theorem
and obtain that
\begin{equation}
\label{ALSO2}\begin{split}
& \iint_{\R\times\R}
\big( v_{\eta,\mu,\e}(x)-v_{\eta,\mu,\e}(y)\big)
\big( \varphi(x)-\varphi(y)\big)\,K(x-y) \,dx\,dy\\
=\;&
\iint_{\R\times\R} \left[\int_\R
\big( v_{\eta,\mu}(x-z)-v_{\eta,\mu}(y-z)\big)\,\vartheta_\e(z)\,
\big( \varphi(x)-\varphi(y)\big)\,K(x-y) \,dz\right]\,dx\,dy\\
=\;&
\int_\R \left[\iint_{\R\times\R}
\big( v_{\eta,\mu}(x-z)-v_{\eta,\mu}(y-z)\big)\,\vartheta_\e(z)\,
\big( \varphi(x)-\varphi(y)\big)\,K(x-y) \,dx\,dy\right]\,dz
\\ =\;&
\int_\R \left[\iint_{\R\times\R}
\big( v_{\eta,\mu}(x)-v_{\eta,\mu}(y)\big)\,\vartheta_\e(z)\,
\big( \varphi(x+z)-\varphi(y+z)\big)\,K(x-y) \,dx\,dy\right]\,dz\\=\;&
\iint_{\R\times\R} \left[\int_\R
\big( v_{\eta,\mu}(x)-v_{\eta,\mu}(y)\big)\,\vartheta_\e(z)\,
\big( \varphi(x+z)-\varphi(y+z)\big)\,K(x-y) \,dz\right]\,dx\,dy\\=\;&
\iint_{\R\times\R} \left[\int_\R
\big( v_{\eta,\mu}(x)-v_{\eta,\mu}(y)\big)\,\vartheta_\e(z)\,
\big( \varphi(x-z)-\varphi(y-z)\big)\,K(x-y) \,dz\right]\,dx\,dy\\=\;&
\iint_{\R\times\R}
\big( v_{\eta,\mu}(x)-v_{\eta,\mu}(y)\big)
\big( \varphi_\e(x)-\varphi_\e(y)\big)\,K(x-y) \,dx\,dy,
\end{split}\end{equation}
where~$\varphi_\e:=\varphi*\vartheta_\e$. 
Similarly, by~\eqref{LE:X2}, we see that
\begin{equation}
\label{ALSO3} \int_\R \dot v_{\eta,\mu,\e}(x)\,\dot \varphi(x)\,dx=
\int_\R \dot v_{\eta,\mu}(x)\,\dot \varphi_\e(x)\,dx.
\end{equation}
{F}rom~\eqref{ALSOIN}, \eqref{ALSO2} and~\eqref{ALSO3}
we infer that, for any~$\varphi\in C^\infty_0(\R,[0,1])$,{\footnotesize{
\begin{eqnarray*}&&\left|
\eta\int_\R \dot v_{\eta,\mu,\e}(x)\,\dot \varphi(x)\,dx+\frac12\,
\iint_{\R\times\R}
\big( v_{\eta,\mu,\e}(x)-v_{\eta,\mu,\e}(y)\big)
\big( \varphi(x)-\varphi(y)\big)\,K(x-y) \,dx\,dy\right|\\
&&\qquad=\left|
\eta\int_\R \dot v_{\eta,\mu}(x)\,\dot \varphi_\e(x)\,dx+\frac12\,
\iint_{\R\times\R}
\big( v_{\eta,\mu}(x)-v_{\eta,\mu}(y)\big)
\big( \varphi_\e(x)-\varphi_\e(y)\big)\,K(x-y) \,dx\,dy\right|\\&&\qquad\le
\kappa\left|\int_\R \varphi_\e(x)\,dx\right|\le\kappa\int_\R \varphi(x)\,dx.
\end{eqnarray*}}}
\noindent That is,
\begin{equation*}
-\kappa\le-\eta \ddot v_{\eta,\mu,\e}+{\mathcal{L}}v_{\eta,\mu,\e}\le\kappa
\end{equation*}
in the sense of distributions, and also in the classical and viscosity senses,
since~$v_{\eta,\mu,\e}$ is smooth.
Therefore, we are in the position of applying
Lemma~\ref{etaregularitylemmaND} to~$v_{\eta,\mu,\e}$ and conclude that
\begin{equation*}\begin{split}
[v_{\eta,\mu,\e}]_{C^{0,\alpha}(\R)}&\le\kappa\big(
1+\|v_{\eta,\mu,\e}\|_{L^\infty(\R)}
\big)^{\frac\alpha{2s}}\|v_{\eta,\mu,\e}\|_{L^\infty(\R)}^{1-\frac{\alpha}{2s}}\\&
\le\kappa\big(
1+\|v_{\eta,\mu}\|_{L^\infty(\R)}
\big)^{\frac\alpha{2s}}\|v_{\eta,\mu}\|_{L^\infty(\R)}^{1-\frac{\alpha}{2s}},
\end{split}\end{equation*}
for any~$\alpha\in(0,2s)$ (up to freely
renaming~$\kappa$). As a consequence of this and~\eqref{LE:X4},
we obtain that~$[v_{\eta,\mu,\e}]_{C^{0,\alpha}(\R)}\le\kappa$.
This and~\eqref{ALSO4} imply that~$[v_{\eta,\mu}]_{C^{0,\alpha}(\R)}\le\kappa$.
Using this and~\eqref{LE:X4}, we obtain that~$
\|v_{\eta,\mu}\|_{C^{0,\alpha}(\R)}\le\kappa$, which in turn implies~\eqref{ALSO5},
as desired.
\end{proof}

\section{Clean intervals and clean points}\label{sec:CI}

Here we deal with the notions of {\em clean intervals}
and {\em clean points}, which have been introduced
in Section~6 of~\cite{MR3594365} to perform glueing techniques
in the nonlocal setting.

\begin{definition}\label{DEF:CLEAN} 
Given~$\rho>0$ and~$Q:\R\to\R$, we say that an interval~$J\subseteq\R$
is a ``clean interval'' for~$(\rho,Q)$
if~$|J|\ge |\log\rho|$ and there exists~$\zeta\in\mathcal{Z}$
such that
$$ \sup_{x\in J}|Q(x)-\zeta|\le\rho.$$
\end{definition}

\begin{definition}\label{DEF:CLEAN:PT}
If~$J$ is a bounded clean interval for~$(\rho,Q)$,
the center of~$J$ is called a ``clean point'' for~$(\rho,Q)$.
\end{definition}

Here we show that any sufficiently large interval
contains a clean interval. 

\begin{lemma}\label{CLEAN:LEMMA}
Let $Q_{\eta,\mu}$ be as in Lemma~\ref{78:77:76}.  Then,  there exist $\rho_0\in(0,1)$ and $\kappa_1>0$ depending on~$Q_{\zeta_1,\zeta_2}^\sharp$
and on the structural constants such that, if $\rho\in (0,\rho_0)$ and 
$J\subseteq\R$ is an
interval such that
\begin{equation}\label{J:ASSU}
|J|\ge 
\frac{\kappa_1[Q_{\eta,\mu}]^\frac{1}{\alpha}_{C^{0,\alpha}(J)}}{ \mu^2\,\rho^{2+\frac{1}{\alpha}}}|\log{\rho}|,\end{equation}
for $\alpha\in(0,2s)$,  then there exists
a clean interval for~$(\rho,Q_{\eta,\mu})$ that is contained in~$J$.
\end{lemma}

\begin{proof} 
By~Corollary \ref{QCalphaobstaclecor}, we know that~$Q_{\eta,\mu}\in C^{0,\alpha}(J)$ for any $\alpha\in(0,2s)$.
Without loss of generality we can assume
that~$[Q_{\eta,\mu}]_{C^{0,\alpha}(J)} \geq 1$.
Assume, by contradiction, that
\begin{equation}\label{90iuojdc876543rg}
{\mbox{$J$ does not contain
any clean subinterval.}} \end{equation}
By~\eqref{J:ASSU}, the interval~$J$
contains~$N$ disjoint subintervals, say~$J_1,\dots,J_N$,
each of length~$|\log\rho|$, with
\begin{equation}\label{98iiogur76yuh}
N\ge \frac{\kappa_1[Q_{\eta,\mu}]^\frac{1}{\alpha}_{C^{0,\alpha}(J)}}{
{\mu^2}\,
\rho^{2+\frac{1}{\alpha}}}-1,\end{equation}
and, by~\eqref{90iuojdc876543rg}, none of the subintervals~$J_i$
is clean. Hence, for any~$i\in\{1,\dots,N\}$, there exists~$p_i\in J_i$
such that~$Q(p_i)$ stays at distance larger than~$\rho$ from  $\mathcal{Z}$. 
Also, letting
$$ \ell_\rho :=
\left(  \frac{\rho}{2  [Q_{\eta,\mu}]_{C^{0,\alpha}(J)}}\right)^\frac{1}{\alpha},$$
we have that,
for any~$x\in J'_i:=[p_i-\ell_\rho,p_i+\ell_\rho]$,
$$ |Q_{\eta,\mu}(x)-Q_{\eta,\mu}(p_i)|\le [Q_{\eta,\mu}]_{C^{0,\alpha}(J)}|x-p_i|^{\alpha}
\le  [Q_{\eta,\mu}]_{C^{0,\alpha}(J)}\, \ell_\rho^{\alpha}
=\frac{\rho}{2}.$$
Accordingly,~$Q_{\eta,\mu}(x)$ stays at distance larger
than~$\frac\rho2$
from $\mathcal{Z}$ for any~$x\in J'_i$ and then,
by~\eqref{POT:GROW},
$$ W(Q_{\eta,\mu}(x))\ge \frac{c_0\,\rho^2}{4}.$$
Moreover, for $\rho$ sufficiently small, at least half of the interval~$J'_i$
lies in~$J_i$, hence
$$ \int_{J_i\cap J_i'} W(Q_{\eta,\mu}(x))\,dx\ge \frac{c_0\,\rho^2\,
\ell_\rho}{4}=\frac{   \kappa \rho^{2+\frac{1}{\alpha}   }}{
[Q_{\eta,\mu}]_{ C^{0,\alpha}(J)}^{\frac{1}{\alpha}  } 
 }.$$
Summing up over~$i=1,\dots,N$, using that the intervals~$J_i$
are disjoint and recalling~\eqref{POT:GROW2},~\eqref{LE:X1}
and~\eqref{LE:X6}, we find that
\begin{equation*}\begin{split}  I_{\eta,\mu}Q_{\zeta_1,\zeta_2}^\sharp&\geq I_{\eta,\mu}(Q_{\eta,\mu})\\&
\ge -\frac\kappa{\mu^2}+\sum_{i=1}^N \int_{J_i\cap J_i'}
a(x)W(Q_{\eta,\mu}(x))\,dx\\&\ge 
-\frac\kappa{\mu^2}+\frac{N\underline{a}\kappa\rho^{2+\frac{1}{\alpha}}}{
[Q_{\eta,\mu}]_{C^{0,\alpha}(J)}^{\frac{1}{\alpha}}    },
\end{split}\end{equation*}
which gives
$$N\leq \frac{\kappa [Q_{\eta,\mu}]_{C^{0,\alpha}(J)}
^{\frac{1}{\alpha}}
}{{\mu^2}\,\rho^{2+\frac{1}{\alpha}}}.$$
This is a contradiction with~\eqref{98iiogur76yuh} for $\kappa_1> \kappa+1$
and so it proves the desired result.
\end{proof}

\begin{lemma}\label{regularitycleaninterval}
Let $Q_{\eta,\mu}$ be as in Lemma~\ref{78:77:76}.
Let~$T>1$ and~$J:=(x_0-4T,x_0+4T)$ be a clean interval
for~$(\rho,Q_{\eta,\mu})$. Then, for any $\alpha\in(0,2s)$, 
$$ [Q_{\eta,\mu} ]_{C^{0,\alpha}(x_0-T,x_0+T)}\le
C \left( \frac{\rho^{1-\frac{\alpha}{2s}}}{|\log{\rho}|^\alpha}
+\rho+{\mu^{\frac{\alpha}{2s}}
\rho^{1-\frac{\alpha}{2s}}}\right)
,$$
for some $C>0$, independent of~$\eta$ and~$\mu$.
\end{lemma}

\begin{proof}
Let $\zeta\in\mathcal{Z}$ be such that $\displaystyle\sup_{x\in J}|Q_{\eta,\mu}(x)-\zeta|\le\rho.$
Then, according to Definition~\ref{DEF:CLEAN}, we have that 
\begin{equation}\label{cleanintervallenght}
T\geq \frac{|\log\rho|}{8},\end{equation} and 
$J\subset F$, where  $F$ is defined as in \eqref{Freeset}. 
Therefore, by Lemma~\ref{HJA:AA:2},  $Q_{\eta,\mu}$ is solution of 
\begin{equation*}
-{\eta,}\,\ddot{Q}_{\eta,\mu}+
\mu(Q_{\eta,\mu}-Q_{\zeta_1,\zeta_2}^\sharp)+{\mathcal{L}}Q_{\eta,\mu}+ a\, W'(Q_{\eta,\mu}) =0
\quad\text{in }J.\end{equation*} 
Then by
Lemma~\ref{IL NUOVO},~\eqref{DAFS823435rt10} and~\eqref{cleanintervallenght}, for~$\alpha<2s$, we have that
\begin{equation*}
\begin{split}
[Q_{\eta,\mu}]_{C^{0,\alpha}(x_0-T,x_0+T)} \,&\le
CT^{-\alpha} \left(
1+T^{2s}\rho+{T^{2s}\mu}\right)^{\frac{\alpha}{2s}}\,\rho^{1-\frac{\alpha}{2s}}\\
&\leq 
CT^{-\alpha} \left(
1+T^{\alpha}\rho^{\frac{\alpha}{2s}}+
{T^{\alpha}\mu^{\frac{\alpha}{2s}}}\right)\,\rho^{1-
\frac{\alpha}{2s}}
\\ &\le
C \left( T^{-\alpha}\rho^{1-\frac{\alpha}{2s}}
+\rho+{\mu^{\frac{\alpha}{2s}}\rho^{1-\frac{\alpha}{2s}}}\right)\\&
\le C\left( \frac{\rho^{1-\frac{\alpha}{2s}}}{|\log{\rho}|^\alpha}
+\rho+{\mu^{\frac{\alpha}{2s}}
\rho^{1-\frac{\alpha}{2s}}}\right),
\end{split}
\end{equation*}
by possibly renaming $C$. This proves the desired estimate
of Lemma~\ref{regularitycleaninterval}.
\end{proof}

\begin{remark}{\rm
Given $x_0\in\R$ and $\beta\in(1,+\infty)$, let  $P:\R\to\R$ be a function such that 
\begin{equation}\label{Pproprem}
v:=P-Q_{\zeta_1,\zeta_2}^\sharp\in H^1(\R)
\end{equation}
and $P$ is H\"older continuous in $(x_0-\beta,x_0+\beta)$, with
\begin{equation}\label{Plip}
[P]_{C^{0,\alpha}(x_0-\beta,x_0+\beta)}\leq\delta,\end{equation}
for some $\delta>0$.
Given $T_1$, $T_2$ such that~$-\infty\leq T_1\leq x_0-\beta<x_0+\beta\leq T_2\leq+\infty$, let us denote 
$$I_-:=(T_1,x_0),\quad I_+:=(x_0,T_2)$$ and 
$$J_-:=(T_1,x_0-\beta),\quad D_-:=(x_0-\beta,x_0),\quad
D_+:=(x_0,x_0+\beta),\quad J_+:=(x_0+\beta,T_2).$$
We want to estimate $E_{(T_1,T_2)^2}(P)$ in terms of $E_{I_-^2}(P)$ and $E_{I_+^2}(P)$.
We have that 
\begin{equation}\label{gliongeqrem1}E_{(T_1,T_2)^2}(P)=
E_{I_-^2}(P)+E_{I_+^2}(P)+2E_{I_-\times I_+}(P)\end{equation}
and 
\begin{equation}\label{gliongeqrem2}
E_{I_-\times I_+}(P)=E_{J_-\times I_+}(P)+E_{D_-\times D_+}(P)+
E_{D_-\times J_+}(P).\end{equation}
By \eqref{Plip} and \eqref{KERNEL}, 
\begin{equation}\label{gliongeqrem3}\begin{split}0\le E_{D_-\times D_+}(P)+[Q_{\zeta_1,\zeta_2}^\sharp]^2_{K,D_-\times D_+}&=\int_{x_0-\beta}^{x_0}\int_{x_0}^{x_0+\beta}
|P(x)-P(y)|^2\,K(x-y)\,dx\,dy\\&
\leq \Theta_0 \delta^2 \int_{x_0-\beta}^{x_0}\int_{x_0}^{x_0+\beta}|x-y|^{
2\alpha-1-2s}\,dx\,dy\\&
\leq \kappa\delta^2 \beta^{2\alpha+1-2s}.
\end{split}
\end{equation}
Moreover, recalling \eqref{Pproprem}, we have that
\begin{equation}\begin{split}\label{iegggdh8575}
E_{J_-\times I_+}(P)&=[v]^2_{K,J_-\times I_+}
+2\B_{J_-\times I_+}(v,Q_{\zeta_1,\zeta_2}^\sharp).
\end{split}
\end{equation}
Now, by \eqref{KERNEL}, {\footnotesize{
\begin{equation*}\begin{split}
&\left|\B_{J_-\times I_+}(v,Q_{\zeta_1,\zeta_2}^\sharp)\right|
=\left|\int_{T_1}^{x_0-\beta}\int_{ x_0}^{T_2}\big(v(x)-v(y)\big)
\big(Q_{\zeta_1,\zeta_2}^\sharp(x)-Q_{\zeta_1,\zeta_2}^\sharp
(y)\big)\,K(x-y)\,dx\,dy\right|\\&\qquad\quad
\leq 2\Theta_0\|Q_{\zeta_1,\zeta_2}^\sharp
\|_{L^\infty(\R)}\int_{-\infty}^{x_0-\beta}
\int_{ x_0}^{+\infty}\big(|v(x)|+|v(y)|\big)|x-y|^{-1-2s}\,dx\,dy\\
&\qquad\quad
=\frac{\Theta_0\|Q_{\zeta_1,\zeta_2}^\sharp\|_{L^\infty(\R)}}{s}
\left[\int_{-\infty}^{x_0-\beta}|v(x)|(x_0-x)^{-2s}\,dx
+\int_{ x_0}^{+\infty}|v(y)|(y-x_0+\beta)^{-2s}\,dy\right].
\end{split}
\end{equation*}}} \noindent
In addition, using the Cauchy-Schwarz inequality and \eqref{RANGE}, we see that 
\begin{equation*}\begin{split}
&\int_{-\infty}^{x_0-\beta}|v(x)|(x_0-x)^{-2s}\,dx\leq\left(\int_{-\infty}^{x_0-\beta}|v(x)|^2\,dx)\right)^\frac{1}{2}\left(\int_{-\infty}^{x_0-\beta}(x_0-x)^{-4s}\,dx\right)^\frac{1}{2}\\&
\leq \kappa\|v\|_{L^2(\R)}\beta^{-\frac{4s-1}{2}}.
\end{split}
\end{equation*}
Similarly, 
\begin{equation*}
\int_{ x_0}^{+\infty}|v(y)|(y-x_0+\beta)^{-2s}\,dy\leq \kappa\|v\|_{L^2(\R)}\beta^{-\frac{4s-1}{2}}.
\end{equation*}
Plugging these pieces of information into~\eqref{iegggdh8575},
we have that
\begin{equation}\label{gliongeqrem4}
\left|E_{J_-\times I_+}(P)\right|\leq [v]^2_{K,(-\infty,x_0-\beta)\times (x_0,+\infty)}+ \kappa\|v\|_{L^2(\R)}\beta^{-\frac{4s-1}{2}}
\end{equation}
Similar computations give
\begin{equation}\label{gliongeqrem5}
\left|E_{D_-\times J_+}(P)\right|\leq [v]^2_{K,(x_0-\beta,x_0)\times (x_0+\beta,+\infty)} +\kappa\|v\|_{L^2(\R)}\beta^{-\frac{4s-1}{2}}
\end{equation}
Hence, from~\eqref{gliongeqrem2},
\eqref{gliongeqrem3},
\eqref{gliongeqrem4} and~\eqref{gliongeqrem5}, we conclude that
\begin{eqnarray*}
&& \left| E_{I_-\times I_+}(P) + [ Q_{\zeta_1,\zeta_2
}^\sharp]^2_{K,(x_0-\beta,x_0)\times (x_0,x_0+\beta)}\right|
\\&\le & \kappa\delta^2 \beta^{2\alpha+1-2s_0}+
\kappa\|v\|_{L^2(\R)}\beta^{-\frac{4s-1}{2}}\\&&\quad
+ [v]^2_{K,(-\infty,x_0-\beta)\times (x_0,+\infty)}+
[v]^2_{K,(x_0-\beta,x_0)\times (x_0+\beta,+\infty)}.
\end{eqnarray*}
This and~\eqref{gliongeqrem1} imply that
\begin{equation}\label{energygluingest}\begin{split}
&\left|E_{(T_1,T_2)^2}(P)-E_{(T_1,x_0)^2}(P)-
E_{(x_0,T_2)^2}(P)+2[ Q_{\zeta_1,\zeta_2
}^\sharp]^2_{K,(x_0-\beta,x_0)\times (x_0,x_0+\beta)}\right|
\\ \le \;&\kappa\delta^2 \beta^{2\alpha+1-2s}+
\kappa\|v\|_{L^2(\R)}\beta^{-\frac{4s-1}{2}}\\&\quad
+ 2[v]^2_{K,(-\infty,x_0-\beta)\times (x_0,+\infty)}+
2[v]^2_{K,(x_0-\beta,x_0)\times (x_0+\beta,+\infty)}.
\end{split}
\end{equation}

Now, thanks to \eqref{energygluingest}, one can consider a clean point~$x_0$ (according to
Definitions~\ref{DEF:CLEAN} and~\ref{DEF:CLEAN:PT})
and glue an optimal trajectory~$Q_{\eta,\mu}$ 
to a linear interpolation with the equilibrium~$\zeta$,
close to~$Q_{\eta,\mu}(x_0)$. Namely, one can consider
\begin{equation}\label{Pgluirem}P(x):=\left\{
\begin{matrix}
Q_{\eta,\mu}(x)  & {\mbox{ if }}x\le x_0,\\
R(x)& {\mbox{ if }} x> x_0,
\end{matrix}
\right. \end{equation}
where $R$ is such that $P- Q_{\zeta_1,\zeta_2}^\sharp
\in H^1(\R)$ and it is defined in~$[x_0,x_0+\beta]$ as follows:
$$R(x):=\left\{
\begin{matrix}
Q_{\eta,\mu}(x_0)\,(x_0+1-x)+\zeta\,(x-x_0) & {\mbox{ if }}x\in
(x_0,x_0+1),\\
\zeta & {\mbox{ if }} x\in [x_0+1,x_0+\beta).\\
\end{matrix}
\right.
$$
In this way, and taking~$\rho>0$ suitably small,
by Definitions~\ref{DEF:CLEAN}
and~\ref{DEF:CLEAN:PT}, we know
that~$Q_{\eta,\mu}$ is $\rho$-close to
an equilibrium in~$[x_0-4\beta,x_0+4\beta]$, with
\begin{equation}\label{098idkscf7gihoj777}
\beta=\beta(\rho)=
\frac{|\log\rho|}{8}.\end{equation}
Moreover, by Lemma~\ref{regularitycleaninterval},
we have that, for $\alpha\in (0,2s)$, 
\begin{equation}\label{Lipstrajecrem}
[Q_{\eta,\mu} ]_{C^{0,\alpha}(x_0-\beta,x_0+\beta)}\le
C \left( \frac{\rho^{1-\frac{\alpha}{2s}}}{|\log{\rho}|^\alpha}
+\rho+{\mu^{\frac{\alpha}{2s}}
\rho^{1-\frac{\alpha}{2s}}}\right)
,\end{equation}
for some $C>0$. Also, we observe that 
$$[R ]_{C^{0,\alpha}(x_0,x_0+\beta)}\le \kappa \rho.$$
As a consequence of this and~\eqref{Lipstrajecrem},
the function~$P$ defined in~\eqref{Pgluirem} satisfies~\eqref{Plip}
with 
\begin{equation}\label{deltalipete}\delta:=
C \left( \frac{\rho^{1-\frac{\alpha}{2s}}}{|\log{\rho}|^\alpha}
+\rho+{\mu^{\frac{\alpha}{2s}}
\rho^{1-\frac{\alpha}{2s}}}\right)\end{equation} 
and $\alpha\in (0,2s)$. 
Thus, choosing~$\beta$
as in~\eqref{098idkscf7gihoj777} and~$\delta$ as in~\eqref{deltalipete}, and recalling~\eqref{LE:X3}
and~\eqref{LE:X5}, we infer from
estimate~\eqref{energygluingest} that
\begin{equation}\label{DIAMOND:EQ}\begin{split}&
\Big|E_{(T_1,T_2)^2}(P)- E_{(T_1,x_0)^2}(Q_{\eta,\mu})\\&-E_{(x_0,T_2)^2}(R)+2[ Q_{\zeta_1,\zeta_2}^\sharp]^2_{K,(x_0-\beta,x_0)\times (x_0,x_0+\beta)}\Big|\leq\diamondsuit,\end{split}
\end{equation}
where we use the notation ``$\diamondsuit$''
to denote quantities that are as small as we wish when~$\rho$
is sufficiently small. The smallness of~$\rho$
depends on~$\mu$,
$Q_{\zeta_1,\zeta_2}^\sharp$
and
the structural constants
of the kernel and the potential, but it is independent on~$\eta$.

We remark that, in virtue of~\eqref{Lipstrajecrem}, we also
have that
\begin{equation}\label{gliongeqrem3Qeta} \begin{split}&
\Big|E_{(T_1,T_2)^2}(Q_{\eta,\mu})- E_{(T_1,x_0)^2}(Q_{\eta,\mu}) \\
& \qquad-E_{(x_0,T_2)^2}(Q_{\eta,\mu})+2[Q_{\zeta_1,\zeta_2}^\sharp]^2_{K,(x_0-\beta,x_0)\times (x_0,x_0+\beta)}\Big|\leq\diamondsuit.
\end{split}\end{equation}
}\end{remark}

\section{Stickiness properties of energy minimizers}\label{MI100}

In this section we show that the minimizers have the tendency to
stick at the minima of~$W$ once they arrive sufficiently close to them.
For this, we recall that~$r\in(0,\,\min\{\delta_0,r_0\}]$
(with~$r_0$ and~$\delta_0$
as in~\eqref{KERNEL} and~\eqref{POT:GROW}, respectively)
has been fixed at the beginning of Section~\ref{S:4}.

\begin{proposition}\label{STICA888}
Let~$\rho\in(0,1)$.
Let~$Q_{\eta,\mu}$ be as in Lemma~\ref{78:77:76}.
Let~$x_1$, $x_2\in\R$ be clean points for~$(\rho,Q_{\eta,\mu})$,
according to Definition~\ref{DEF:CLEAN:PT},
with~$x_2\ge x_1+4$, and
\begin{equation}\label{ALKKA} \max_{i=1,2} 
|Q_{\eta,\mu}(x_i)-\zeta|\le\rho,\end{equation}
for some~$\zeta\in\mathcal{Z}$. 
Then
\begin{equation}
\label{STIMA:AA1}\begin{split}&\frac{\eta}{2}\int_{x_1}^{x_2}|\dot{Q}_{\eta,\mu}(x)|^2\,dx+ 
\frac\mu2\int_{x_1}^{x_2} |Q_{\eta,{\mu}}(x)-Q^\sharp_{\zeta_1,\zeta_2}(x)|^2\,dx
\\&\qquad+
\frac14[Q_{\eta,\mu}]^2_{K,(x_1,x_2)^2}+
\int_{x_1}^{x_2} a(x)\,W(Q_{\eta,\mu}(x))\,dx\le \diamondsuit
,\end{split}\end{equation}
with~$\diamondsuit$ as small as we wish if~$\rho$ is suitably small
(the smallness of~$\rho$
depends on $\mu$, $Q_{\zeta_1,\zeta_2}^\sharp$,
and
on structural constants, but it is independent on~$\eta$).

Moreover,
\begin{equation}
\label{STIMA:AA2}
{\mbox{$|Q_{\eta,\mu}(x)-\zeta|\le \displaystyle\frac{r}2$ for
every~$x\in [x_1,x_2]$.}}\end{equation}
\end{proposition}

\begin{proof} First of all, we observe that, if~$\rho$
is sufficiently small with respect to~$\mu$, then
\begin{equation}\label{STan7234}
{\mbox{either $x_2\le-2$ or $x_1\ge2$.}}
\end{equation}
Indeed, suppose not, then~$x_2>-2$ and~$x_1<2$.
Hence, if~$x\in\left[ x_2+3,x_2+\frac{|\log\rho|}2\right]$,
we have that~$x\ge1$ and thus~$Q^\sharp_{\zeta_1,\zeta_2}(x)=\zeta_2$.
Consequently, recalling that~$|Q_{\eta,{\mu}}(x)-\zeta|\le\rho$
for all~$x\in\left[ x_2+3,x_2+\frac{|\log\rho|}2\right]$, thanks to
Definition~\ref{DEF:CLEAN}, we find that
$$ |Q_{\eta,{\mu}}(x)-Q^\sharp_{\zeta_1,\zeta_2}(x)|=
|Q_{\eta,{\mu}}(x)-\zeta_2|\ge|\zeta_2-\zeta|-\rho$$
for all~$x\in\left[ x_2+3,x_2+\frac{|\log\rho|}2\right]$.

Similarly, we have that
$$ |Q_{\eta,{\mu}}(x)-Q^\sharp_{\zeta_1,\zeta_2}(x)|\ge|\zeta_1-\zeta|-\rho$$
for all~$x\in\left[ x_1-\frac{|\log\rho|}2,x_1-3\right]$.

In particular, since either~$\zeta\ne\zeta_1$ or~$\zeta\ne\zeta_2$, we conclude that
$$ |Q_{\eta,{\mu}}(x)-Q^\sharp_{\zeta_1,\zeta_2}(x)|\ge c_0,$$
for some~$c_0>0$, for all~$x$ belonging to
an interval of length~$\frac{|\log\rho|}2-3\ge
\frac{|\log\rho|}4$.

For these reasons, we conclude that
$$ \|Q_{\eta,{\mu}}-Q^\sharp_{\zeta_1,\zeta_2}\|_{L^2(\R)}^2\ge
\frac{c_0^2\,|\log\rho|}4.
$$
This and~\eqref{LE:X5} yield that
$$
\frac{\kappa^2}{\mu^2}\ge\|Q_{\eta,{\mu}}-Q^\sharp_{\zeta_1,\zeta_2}\|_{L^2(\R)}^2\ge
\frac{c_0^2\,|\log\rho|}4,
$$
which is a contradiction if~$\rho$ is sufficiently small (possibly in dependence of~$\mu$).
This proves~\eqref{STan7234}.

Hence, in the light of~\eqref{STan7234}, without loss of generality we can now suppose
that
\begin{equation}\label{ujninsdfl234r}
x_1\ge2.\end{equation} We claim that, in this case,
\begin{equation}\label{sonufj123}
\zeta=\zeta_2.
\end{equation}
Indeed, suppose not. Then~$|\zeta-\zeta_2|\ge \hat{c}$, for some~$\hat{c}>0$.
Also, if~$x\ge x_1$ we have that~$Q^\sharp_{\zeta_1,\zeta_2}(x)=\zeta_2$,
and consequently, for all~$x\in\left[ x_1,x_1+\frac{|\log\rho|}2\right]$,
$$ |Q_{\eta,{\mu}}(x)-Q^\sharp_{\zeta_1,\zeta_2}(x)|=|Q_{\eta,{\mu}}(x)-\zeta_2|\ge
|\zeta-\zeta_2|-|Q_{\eta,{\mu}}(x)-\zeta|\ge\hat{c}-\rho\ge\frac{\hat{c}}2,
$$
if~$\rho$ is small enough.

As a result,
$$ \|Q_{\eta,{\mu}}-Q^\sharp_{\zeta_1,\zeta_2}\|_{L^2(\R)}^2\ge
\frac{\hat{c}^2\,|\log\rho|}8.$$
This and~\eqref{LE:X5} yield that
$$
\frac{\kappa^2}{\mu^2}\ge\frac{\hat{c}^2\,|\log\rho|}8,$$
from which we obtain a contradiction if~$\rho$ is small enough.
The proof of~\eqref{sonufj123} is thereby complete.

Now, in light of~\eqref{sonufj123}, we define
$$ P(x):= \left\{
\begin{matrix}
Q_{\eta,\mu}(x) & {\mbox{ if }} x\in (-\infty, x_1),\\
Q_{\eta,\mu}(x_1)(x_1+1-x)+\zeta_2(x-x_1) & {\mbox{ if }} x\in [x_1,x_1+1],\\
\zeta_2 & {\mbox{ if }} x\in (x_1+1, x_2-1),\\
Q_{\eta,\mu}(x_2)(x-x_2+1)+\zeta_2(x_2-x) & {\mbox{ if }} x\in [x_2-1,x_2],\\
Q_{\eta,\mu}(x)& {\mbox{ if }} x\in (x_2,+\infty).
\end{matrix}
\right. $$
In this way, we have that
\begin{equation}\label{normP}
[P]_{C^{0,1}(x_1,x_2)}\leq\rho.
\end{equation}
Also, in view of~\eqref{ujninsdfl234r},
\begin{equation}\label{797jhneutr0000}\begin{split}
\int_{x_1}^{x_2} |P(x)-Q^\sharp_{\zeta_1,\zeta_2}(x)|^2\,dx&=
\int_{x_1}^{x_2} |P(x)-\zeta_2|^2\,dx\\&=
\int_{[x_1,x_1+1]\cup[x_2-1,x_2]} |P(x)-\zeta_2|^2\,dx.
\end{split}\end{equation}
Moreover, we observe that, if~$x\in(x_1,x_2)$, then
\begin{equation}\label{89iok66670a0}
\begin{split}
&|P(x)-\zeta_2|\\
\le\;& \sup_{y\in(x_1,x_1+1)} |Q_{\eta,\mu}(x_1)(x_1+1-y)
+\zeta_2(y-x_1)
-\zeta_2|\\&\qquad\qquad+
\sup_{y\in(x_2-1,x_2)} |Q_{\eta,\mu}(x_2)(y-x_2-1)
+\zeta_2(x_2-y)
-\zeta_2|\\ \le\;&
|Q_{\eta,\mu}(x_1)-\zeta_2|+|Q_{\eta,\mu}(x_2)-\zeta_2|\\
\le\;& 2\rho,
\end{split}\end{equation}
thanks to~\eqref{ALKKA} and~\eqref{sonufj123}.

Therefore, plugging this information into~\eqref{797jhneutr0000},
$$ \int_{x_1}^{x_2} |P(x)-Q^\sharp_{\zeta_1,\zeta_2}(x)|^2\,dx\le2\rho^2.$$
For that reason,
\begin{equation}\label{lasndf706574}\begin{split}& \int_{\R} |P(x)-Q^\sharp_{\zeta_1,\zeta_2}(x)|^2\,dx-
\int_{\R} |Q_{\eta,{\mu}}(x)-Q^\sharp_{\zeta_1,\zeta_2}(x)|^2\,dx\\
=\;&
\int_{x_1}^{x_2} |P(x)-Q^\sharp_{\zeta_1,\zeta_2}(x)|^2\,dx-
\int_{x_1}^{x_2} |Q_{\eta,{\mu}}(x)-Q^\sharp_{\zeta_1,\zeta_2}(x)|^2\,dx\\
\le\;&2\rho^2-\int_{x_1}^{x_2} |Q_{\eta,{\mu}}(x)-Q^\sharp_{\zeta_1,\zeta_2}(x)|^2\,dx.\end{split}\end{equation}
Also,
\begin{equation}\label{gherugh-8436h}
{\mbox{if $x,y\in(x_1,x_2)$, then
$|P(x)-P(y)|\leq 2\rho$.}}\end{equation} 
Now, let us estimate $[P]^2_{K,(x_1,x_2)^2}$. We have 
\begin{equation}\label{linearest}\begin{split}
[P]^2_{K,(x_1,x_2)^2}&=[P]^2_{K,(x_1,x_1+1)\times(x_1,x_2)}\\&\quad+[P]^2_{K,(x_1+1,x_2-1)\times(x_1,x_2)}+[P]^2_{K,(x_2-1,x_2)\times(x_1,x_2)}.
\end{split}
\end{equation}
Using~\eqref{KERNEL}, \eqref{normP} and~\eqref{gherugh-8436h},
we see that 
\begin{equation}\label{linearest2}\begin{split}
& [P]^2_{K,(x_1,x_1+1)\times(x_1,x_2)}\\=\;&
\int_{x_1}^{x_1+1}\int_{x_1}^{x_1+2}|P(x)-P(y)|^2\,K(x-y)\,dx\,dy\\&\qquad
+\int_{x_1}^{x_1+1}\int_{x_1+2}^{x_2}|P(x)-P(y)|^2\,K(x-y)
\,dx\,dy\\
\le\;&\Theta_0 \rho^2 \int_{x_1}^{x_1+1}\int_{x_1}^{x_1+2}
|x-y|^{1-2s}\,dx\,dy
+4\Theta_0 \rho^2\int_{x_1}^{x_1+1}\int_{x_1+2}^{x_2}
|x-y|^{-1-2s}\,dx\,dy\\
\leq\; & \kappa\rho^2 \\
=\;&\diamondsuit.
\end{split}
\end{equation}
Similarly,
\begin{equation}\label{linearest3}
[P]^2_{K,(x_2-1,x_2)\times(x_1,x_2)}\leq \diamondsuit.
\end{equation}
Finally, making again use of~\eqref{KERNEL},
\eqref{normP} and~\eqref{gherugh-8436h}, we compute
\begin{equation}\label{linearest4}\begin{split}
& [P]^2_{K(x_1+1,x_2-1)\times(x_1,x_2)}\\
=\;&\int_{x_1+1}^{x_2-1}\int_{x_1}^{x_1+1}|P(x)-P(y)|^2\,K(x-y)
\,dx\,dy\\
&\quad
+\int_{x_1+1}^{x_2-1}\int_{x_2-1}^{x_2}|P(x)-P(y)|^2\,K(x-y)\,dx\,dy\\
=\;&\int_{x_1+1}^{x_1+2}\int_{x_1}^{x_1+1}|P(x)-P(y)|^2
\,K(x-y)\,dx\,dy\\
&\quad
+\int_{x_1+2}^{x_2-1}\int_{x_1}^{x_1+1}|P(x)-P(y)|^2
\,K(x-y)\,dx\,dy\\
&\quad +\int_{x_1+1}^{x_2-2}\int_{x_2-1}^{x_2}|P(x)-P(y)|^2
\,K(x-y)\,dx\,dy\\
&\quad
+\int_{x_2-2}^{x_2-1}\int_{x_2-1}^{x_2}|P(x)-P(y)|^2\,K_m(x-y)
\,dx\,dy\\
\leq\;& \kappa \rho^2\ \Big(\int_{x_1+1}^{x_1+2}
\int_{x_1}^{x_1+1}|x-y|^{1-2s}\,dx\,dy+
\int_{x_2-2}^{x_2-1}\int_{x_2-1}^{x_2}|x-y|^{1-2s}\,dx\,dy
\\&\qquad\quad
+\int_{x_1+2}^{x_2-1}\int_{x_1}^{x_1+1}|x-y|^{-1-2s}
\,dx\,dy+\int_{x_1+1}^{x_2-2}\int_{x_2-1}^{x_2}
|x-y|^{-1-2s}\,dx\,dy\Big)\\
\leq\;&\kappa\rho^2\\
=\;&\diamondsuit.
\end{split}
\end{equation}
Therefore, collecting estimates \eqref{linearest}, \eqref{linearest2},
\eqref{linearest3}
and~\eqref{linearest4}, we get
\begin{equation}\label{linearest5}
[P]^2_{K,(x_1,x_2)^2}\leq \diamondsuit.
\end{equation}
Combining~\eqref{DIAMOND:EQ}
(applied here twice, with~$x_0:=x_1$ and~$x_0:=x_2$)
with~\eqref{linearest5} yields, for~$\beta$ as
in~\eqref{098idkscf7gihoj777},
\begin{equation}\label{97yihokjb6666666:000}\begin{split}
E_{\R^2}(P)&\le E_{(-\infty,x_1)^2}(Q_{\eta,\mu})+E_{(x_1,x_2)^2}(P)+E_{(x_2,+\infty)^2}(Q_{\eta,\mu})+\diamondsuit
\\&\qquad -2[ Q_{\zeta_1,\zeta_2}^\sharp]_{K,(x_1-\beta,x_1)\times(x_1,x_1+\beta)}^2-2[
Q_{\zeta_1,\zeta_2}^\sharp]_{K,(x_2-\beta,x_2)\times(x_2,x_2+\beta)}^2
\\&=E_{(-\infty,x_1)^2}(Q_{\eta,\mu})+E_{(x_2,+\infty)^2}(Q_{\eta,\mu})
+\diamondsuit
-[Q_{\zeta_1,\zeta_2}^\sharp]_{K,(x_1,x_2)^2}^2
\\&\qquad -2[Q_{\zeta_1,\zeta_2}^\sharp
]_{K,(x_1-\beta,x_1)\times(x_1,x_1+\beta)}^2-2[ Q_{\zeta_1,\zeta_2}^\sharp]_{K,(x_2-\beta,x_2)\times(x_2,x_2+\beta)}^2.
\end{split}\end{equation}
On the other hand, by \eqref{gliongeqrem3Qeta}
(again applied here twice, with~$x_0:=x_1$ and~$x_0:=x_2$),
we have that
\begin{equation}\label{97yihokjb6666666:000bis}\begin{split}
E_{\R^2}(Q_{\eta,\mu})&\ge  E_{(-\infty,x_1)^2}(Q_{\eta,\mu})+E_{(x_1,x_2)^2}(Q_{\eta,\mu})+E_{(x_2,+\infty)^2}(Q_{\eta,\mu})+\diamondsuit
\\&\qquad -2[ Q_{\zeta_1,\zeta_2}^\sharp]_{K,(x_1-\beta,x_1)\times(x_1,x_1+\beta)}^2-2[ Q_{\zeta_1,\zeta_2}^\sharp]_{K,(x_2-\beta,x_2)\times(x_2,x_2+\beta)}^2.
\end{split}\end{equation}
Subtracting \eqref{97yihokjb6666666:000bis}
to~\eqref{97yihokjb6666666:000}, we get
\begin{equation}\label{97yihokjb6666666:000bisbis}
E_{\R^2}(P)-E_{\R^2}(Q_{\eta,\mu})\leq -[Q_{\eta,\mu}]_{K(x_1,x_2)^2}^2+\diamondsuit.
\end{equation}
In addition, by~\eqref{POT:GROW}
and~\eqref{89iok66670a0}, we see that if~$x\in (x_1,x_2)$
then~$W(P(x))\le 4C_0\rho^2$.
Using this and the fact that~$W(P(x))=W(\zeta)=0$
if~$x\in(x_1+1,x_2-1)$, we conclude that
$$ \int_{x_1}^{x_2} W(P(x))\,dx=
\int_{x_1}^{x_1+1} W(P(x))\,dx+\int_{x_2-1}^{x_2} W(P(x))\,dx
\le 8C_0\,\rho^2.$$
Thus, by the minimality of~$Q_{\eta,\mu}$ for~$I_{\eta,\mu}$ 
(defined in~\eqref{01edkIETA}), \eqref{lasndf706574}
and~\eqref{97yihokjb6666666:000bisbis},
\begin{eqnarray*}
0&\le&I_{\eta,\mu}(P)-I_{\eta,\mu}(Q_{\eta,\mu})\\ &\le& 
\eta\rho-\frac{\eta}{2}\int_{x_1}^{x_2}|\dot{Q}_{\eta,\mu}(x)|^2\,dx
+\mu\rho^2-\frac\mu2\int_{x_1}^{x_2} |Q_{\eta,{\mu}}(x)-Q^\sharp_{\zeta_1,\zeta_2}(x)|^2\,dx
\\&&\qquad
-\frac14[Q_{\eta,\mu}]_{K,(x_1,x_2)^2}^2-\int_{x_1}^{x_2} a(x)\,W(Q_{\eta,\mu}(x))\,dx+\diamondsuit,
\end{eqnarray*}
which proves~\eqref{STIMA:AA1}.

Now we prove~\eqref{STIMA:AA2}.
For this,
we assume by contradiction
that there exists~$\tilde x\in[x_1,x_2]$
such that~$|Q_{\eta,\mu}(\tilde x)-\zeta|> r/2$. 

By  Corollary~\ref{QCalphaobstaclecor},  we
have that~$Q_{\eta,\mu}$ is H\"{o}lder continuous (with uniform bound).
Hence, since~$|Q_{\eta,\mu}(x_1)-\zeta|\le\rho <r/2$,
we obtain that there exists~$\hat x\in[x_1,x_2]$
such that
\begin{equation}\label{9asdwAA}
|Q(\hat x)-\zeta|= \frac{r}2.\end{equation}
In particular, there exists~$\ell$ independent of~$\eta$
and~$\mu$ such that,
for any~$x\in [ \hat x -\ell , \hat x + \ell]$ and~$\alpha\in(0,2s)$, 
$$ |Q_{\eta,\mu}(x)-Q_{\eta,\mu}(\hat x)|\le \kappa\,|x-\hat x|^{\alpha}\le \frac{r}{4}.$$
This and~\eqref{9asdwAA}
imply that, if~$x\in [ \hat x -\ell , \hat x + \ell]$,
$$ Q_{\eta,\mu}(x) \in \overline{ B_{3r/4} (\zeta)\setminus B_{r/4}(\zeta)}$$
and thus
$$ {\rm dist}\,\big(Q_{\eta,\mu}(x),\mathcal{Z}\big)\ge \frac{r}{4},$$
for all~$x\in [ \hat x -\ell, \hat x + \ell]$.
This,~\eqref{POT:GROW} and~\eqref{POT:GROW2} give that
\begin{equation*}
\int_{\hat x -\ell}^{\hat x + \ell} a(x) W(Q_{\eta,\mu}(x))\,dx\ge \underline{a}
\int_{\hat x -\ell}^{\hat x + \ell} W(Q_{\eta,\mu}(x))\,dx \ge
2\ell\,\underline{a} \,
\inf_{ {\rm dist}\, (\tau,\; \mathcal{Z}) \ge r/4 } W(\tau)=:c.\end{equation*}
Hence, noticing that~$(\hat x-\ell,\hat x+\ell)\subseteq (x_1,x_2)$,
we obtain that
$$ \int_{x_1}^{x_2} a(x) W(Q_{\eta,\mu}(x))\,dx\ge c,$$
and this is in contradiction with~\eqref{STIMA:AA1}
for small~$\rho$. Then, the proof
of~\eqref{STIMA:AA2} is now complete.
\end{proof}

\section{Reduction to the case in which~${\mathcal{Z}}=\{\zeta_1,\zeta_2\}$}\label{REDUSE}

Now we remark that it is enough to prove Theorem~\ref{MAIN}
under the additional assumption that
\begin{equation}\label{67:Asdf3456}
{\mathcal{Z}}=\{\zeta_1,\zeta_2\}.
\end{equation}
Indeed, let~$\zeta_1\ne\zeta_2\in{\mathcal{Z}}$ be
nearest neighbors according to Definition~\ref{PRIMICI},
and suppose, without loss of generality, that~$\zeta_1<\zeta_2$.
Then, we can find~$\delta>0$ sufficiently small such that~$
{\mathcal{Z}}^\star:=
[\zeta_1-\delta,
\zeta_2+\delta]\cap{\mathcal{Z}}=\{\zeta_1,\zeta_2\}$.
We also consider a potential~$W^\star$ which coincides with~$W$
in~$[\zeta_1-\delta,
\zeta_2+\delta]$ and remains strictly positive outside~$[\zeta_1-\delta,
\zeta_2+\delta]$.

In this way, the potential~$W^\star$ satisfies the same structural
assumptions of~$W$ for a set of equilibria~${\mathcal{Z}}^\star$
as in~\eqref{67:Asdf3456}. Then, if
Theorem~\ref{MAIN} holds true under the additional assumption~\eqref{67:Asdf3456},
we obtain a heteroclinic orbit~$Q^\star:\R\to[\zeta_1,\zeta_2]$ connecting~$\zeta_1$ to~$\zeta_2$,
which is a solution of~${\mathcal{L}}Q^\star+a\,{W^\star}\,'(Q^\star)=0$.
But then, since~$W^\star=W$ in the range of~$Q^\star$,
we obtain that~$Q^\star$ is also a solution of~\eqref{EQUAZIONE},
thus giving Theorem~\ref{MAIN} in its full generality.

For that reason, from now on, we assume without loss of generality
that condition~\eqref{67:Asdf3456}
is also satisfied.

\section{Unconstrained minimization for a perturbed problem}\label{UNCO}

Here, recalling the setting of Section~\ref{S:4},
we show that if~$b_1$ and~$b_{2}$ are sufficiently separated, then the constrained
minimizer, whose existence has been established in Lemma~\ref{78:77:76},
is in fact an unconstrained minimizer. The idea for this is
that the ``excursion'' of the minimizer will occur
at the points ``favored by the wells of~$a$''
(recall the non-degeneracy condition in~\eqref{a nondegenerate}), which can be placed
suitably far from the constraints.

Under the additional assumption in~\eqref{67:Asdf3456},
we consider the minimizer~$Q_{\eta,\mu}$
for~$I_{\eta,\mu}$
as given  in Lemma~\ref{78:77:76}. In this setting, we have:

\begin{proposition}\label{FREE-eta}
There exists a structural constant~$\mu_0>0$
such that if~$\mu\in[0,\mu_0]$ the following statement holds true.

There exist~$b_1$, $b_2\in\R$
and
\begin{equation}\label{PASerf6465-1}
Q_{\eta,\mu}^\star\in\Gamma(b_1,b_2)\end{equation} such that
\begin{equation}\label{6cashjdcocn}
Q_{\eta,\mu}^\star:\R\to
\big[\min\{\zeta_1,\zeta_2\},\max\{\zeta_1,\zeta_2\}\big]\end{equation} and
\begin{equation}
\label{LE:X1-FREE} I_{\eta,\mu}(Q_{\eta,\mu}^\star)\le I_{\eta,\mu}(Q) {\mbox{ for all }} Q {\mbox{ s.t. }}
Q-Q_{\zeta_1,\zeta_2}^\sharp\in H^1(\R)
.\end{equation}
Also, letting~$v_{\eta,\mu}^\star:=Q_{\eta,\mu}^\star-Q_{\zeta_1,\zeta_2}^\sharp$, we see that
\begin{eqnarray}
\label{LE:X2-FREE} && [v^\star_{\eta,\mu}]_{H^1(\R)}\le\frac{\kappa}{\sqrt{\eta{\mu}}},\\
\label{LE:X3-FREE} && [v^\star_{\eta,\mu}]_{K,\R\times\R}\le\frac\kappa{\sqrt{\mu}},\\
\label{LE:X4-FREE} && \|v_{\eta,\mu}^\star\|_{L^\infty(\R)}\le\kappa,\\
\label{LE:X5-FREE}&&
\|v_{\eta,\mu}^\star\|_{L^2(\R)}\le\frac\kappa{{\mu}} \\
\label{LE:X6-FREE} {\mbox{and }}&&
\|v_{\eta,\mu}^\star\|_{C^{0,\alpha}(\R)}\le\kappa \text{ for all }\alpha\in (0,2s),
\end{eqnarray}
for some~$\kappa>0$, which
possibly depends on~$Q_{\zeta_1,\zeta_2}^\sharp$
and on structural constants.
\end{proposition}

\begin{proof} We stress that the main difference between~\eqref{LE:X1} and~\eqref{LE:X1-FREE}
is that the competitors in~\eqref{LE:X1-FREE} do not need to be in~$\Gamma(b_1,b_2)$
and so~$Q_{\eta,\mu}^\star$ is a free minimizer. The proof of Proposition~\ref{FREE-eta} is a slight modification of the proof of Theorem~9.4
in~\cite{MR3594365},
and we refer to it   for more details.

Let~$Q_{\eta,\mu}^\star$ be as in Lemma~\ref{78:77:76} and 
let $v_{\eta,\mu}^\star:=Q_{\eta,\mu}^\star- Q_{\zeta_1,\zeta_2}^\sharp$. Then by  Lemma~\ref{78:77:76}
and Corollary~\ref{QCalphaobstaclecor}  
we have that~$v_{\eta,\mu}^\star$ satisfies~\eqref{LE:X2-FREE}--\eqref{LE:X6-FREE}. 

To prove~\eqref{LE:X1-FREE},  
we fix~$\rho\in(0,r)$, to be taken sufficiently
small, possibly in dependence of~$\mu$,
and we set
$$b_1=m_1\quad\text{and}\quad b_2=m_2,$$
with $m_1,m_2$ given by~\eqref{a nondegenerate}. 
To prove Proposition \ref{FREE-eta}, we want to show that $Q_{\eta,\mu}^\star$ does not touch the constraints of $\Gamma(b_1,b_2)$. Assume by contradiction that 
\begin{equation}\label{contradicexthm}
\text{there exists }x_1\leq b_1=m_1
\text{ such that either }Q_{\eta,\mu}^\star(x_1)=
\Phi(x_1)\text{ or }Q_{\eta,\mu}^\star(x_1)=\Psi(x_1),\end{equation}
the other case being similar.
In particular, by~\eqref{Phidef}
and~\eqref{Psidef}, we have that~$
|Q_{\eta,\mu}^\star(x_1)-\zeta_1|\geq \frac34r$.
Also, by~\eqref{LE:X6-FREE}, we know that~$
[Q_{\eta,\mu}^\star]_{ C^{0,\alpha}(\R)}\le\kappa$,
for $\alpha\in(0,2s)$.
Thus, by Lemma~\ref{CLEAN:LEMMA}, if 
\begin{equation*}\omega\geq
\frac{\kappa_1\kappa^\frac1\alpha}{{\mu^2}\,\rho^{2+\frac1\alpha}}
|\log{\rho}|+1,\end{equation*}
we conclude that
\begin{equation}\label{cleanpointx_*}\begin{split}&
\text{there exist a clean point }
x_*\in(m_1+1,m_1+\omega)
\text{ and }\zeta\in\mathcal{Z}\\&\text{such that }
Q_{\eta,\mu}^\star(x_*)\in \overline{B_{\rho}(\zeta)}.\end{split}\end{equation}
Furthermore, by~\eqref{67:Asdf3456}, we have that~$\zeta\in\{\zeta_1,\zeta_2\}$. Now, arguing as
in the proof of Theorem~9.4 in~\cite{MR3594365} (see in particular
the comments between~(9.12) and~(9.13) in~\cite{MR3594365}),
and using \eqref{contradicexthm},
we see that we must actually have that 
\begin{equation}\label{cleanpointx_*bis}\zeta=\zeta_2\end{equation}
and that $Q_{\eta,\mu}^\star(x)\in  \overline{B_{\frac{r}{2}}}(\zeta_2)$ for any $x\geq x_*$. In  particular, since by~\eqref{m1m2theta},  $x_*\leq m_1+\omega\leq m_2-\theta$,  we have that 
\begin{equation}\label{qstarafterx*} Q_{\eta,\mu}^\star(x)\in  \overline{B_{\frac{r}{2}}}(\zeta_2)\text{ for any }x\geq m_2-\theta.\end{equation}
Now we define $P(x):= Q_{\eta,\mu}^\star(x-\theta)$. 
We remark that if~$x\ge1+\theta$ then~$Q_{\zeta_1,\zeta_2}^\sharp(x)=\zeta_2=
Q_{\zeta_1,\zeta_2}^\sharp(x-\theta)$. Similarly, if~$x\le-1$,
then~$Q_{\zeta_1,\zeta_2}^\sharp(x)=\zeta_1=
Q_{\zeta_1,\zeta_2}^\sharp(x-\theta)$. As a result, we see that
the function~$x\mapsto Q_{\zeta_1,\zeta_2}^\sharp(x-\theta)-Q_{\zeta_1,\zeta_2}^\sharp(x)$
vanishes outside~$[-1,1+\theta]$ and therefore
\begin{equation}\label{ESSAnuro9234}
\begin{split}&
\int_\R \Big( |P(x)-Q_{\zeta_1,\zeta_2}^\sharp(x)|^2-|
Q_{\eta,\mu}^\star(x)-
Q_{\zeta_1,\zeta_2}^\sharp(x)|^2\Big)\,dx\\=\;&
\int_\R \Big( |Q_{\eta,{\mu}}^\star(x-\theta)-Q_{\zeta_1,\zeta_2}^\sharp(x)|^2-|
Q_{\eta,\mu}^\star(x)-
Q_{\zeta_1,\zeta_2}^\sharp(x)|^2\Big)\,dx
\\=\;&
\int_\R \Big( |v_{\eta,{\mu}}^\star(x-\theta)+
Q_{\zeta_1,\zeta_2}^\sharp(x-\theta)
-Q_{\zeta_1,\zeta_2}^\sharp(x)|^2-|
v_{\eta,\mu}^\star(x)|^2\Big)\,dx\\=\;&
\int_\R \Big( |v_{\eta,{\mu}}^\star(x-\theta)|^2-|
v_{\eta,\mu}^\star(x)|^2+2v_{\eta,{\mu}}^\star(x-\theta)
\big(Q_{\zeta_1,\zeta_2}^\sharp(x-\theta)
-Q_{\zeta_1,\zeta_2}^\sharp(x)\big)\\&\qquad\qquad+
|Q_{\zeta_1,\zeta_2}^\sharp(x-\theta)
-Q_{\zeta_1,\zeta_2}^\sharp(x)|^2\Big)\,dx\\=\;&
\int_\R \Big( 2v_{\eta,{\mu}}^\star(x-\theta)
\big(Q_{\zeta_1,\zeta_2}^\sharp(x-\theta)
-Q_{\zeta_1,\zeta_2}^\sharp(x)\big)+
|Q_{\zeta_1,\zeta_2}^\sharp(x-\theta)
-Q_{\zeta_1,\zeta_2}^\sharp(x)|^2\Big)\,dx\\ \leq\;&\const
\int_\R 
|Q_{\zeta_1,\zeta_2}^\sharp(x-\theta)
-Q_{\zeta_1,\zeta_2}^\sharp(x)|\,dx\\=\;&\const
\int_{-1}^{1+\theta}
|Q_{\zeta_1,\zeta_2}^\sharp(x-\theta)
-Q_{\zeta_1,\zeta_2}^\sharp(x)|\,dx\\ \le\;&\const.
\end{split}
\end{equation}
Also, due to \eqref{qstarafterx*}, we have that $P\in \Gamma(b_1,b_2)$ and therefore,
by~\eqref{ESSAnuro9234}
and the minimality of~$Q_{\eta,\mu}^\star$,
\begin{equation}\label{9ytguijh}
\begin{split}
 0& \le I_{\eta,\mu}(P)-I_{\eta,\mu}(Q_{\eta,\mu}^\star)\\
 &=
\int_\R a(x)\, W(P(x))\,dx -\int_\R a(x)\, W(Q_{\eta,\mu}^\star(x))\,dx\\
&\qquad +\frac{\mu}2
\int_\R \Big( |P(x)-Q_{\zeta_1,\zeta_2}^\sharp(x)|^2-|
Q_{\eta,\mu}^\star(x)-
Q_{\zeta_1,\zeta_2}^\sharp(x)|^2\Big)\,dx
\\ &\le
\int_\R a(x)\, W(Q_{\eta,\mu}^\star(x-\theta
))\,dx -\int_\R a(x)\, W(Q_{\eta,\mu}^\star(x))\,dx+
{\const\mu}
\\ &=
\int_\R \big[ a(x+\theta)-a(x) \big]\, W(Q_{\eta,\mu}^\star(x))\,dx+{\const\mu}.
\end{split}\end{equation}
Now, we observe that $Q_{\eta,\mu}^\star(m_1)\in
\overline{B_{\frac54 r}(\zeta_1)}$ and~$
Q_{\eta,\mu}^\star(x_*)\in \overline{B_\rho(\zeta_2)}$, due to \eqref{cleanpointx_*} and \eqref{cleanpointx_*bis}.
Therefore, since ~$Q_{\eta,\mu}^*$ is continuous, there exists~$y_*\in (m_1,m_1+\omega)$  such that either~$Q_{\eta,\mu}^*(y_*)=\zeta_{1}+\frac12$ 
or~$Q_{\eta,\mu}^*(y_*)=\zeta_{1}-\frac12$. We
assume without loss of generality that~$Q_{\eta,\mu}^*(y_*)
=\zeta_{1}+\frac12$.
Then by the H\"{o}lder continuity of~$Q_{\eta,\mu}^\star$, there exists 
an interval~$J_*\subset(m_1,m_1+\omega)$
of uniform length and centered at~$y_*$ such that~$
Q_{\eta,\mu}^\star(x)$
stays at distance~$1/4$ from~$\mathcal{Z}$ for any~$x\in J_*$. 
Therefore, using~\eqref{a nondegenerate}, we get
\begin{equation}\label{pricipalintextthm}\begin{split}
&\int_{m_1-\omega}^{m_1+\omega} \big[ a(x+\theta)-a(x) \big]\, W(Q_{\eta,\mu}^\star(x))\,dx\leq \int_{J_*} \big[ a(x+\theta)-a(x) \big]\, W(Q_{\eta,\mu}^\star(x))\,dx\\&
\qquad\leq -\gamma \int_{J_*} W(Q_{\eta,\mu}^\star(x))\,dx\leq -\tilde{\gamma}\inf_{\text{dist}(\tau,\mathcal{Z})\ge1/4}=:-\hat{\gamma}.
\end{split}\end{equation} 
Then, by~\eqref{LE:X5-FREE} and the continuity of~$Q_{\eta,\mu}^\star$,
we know that there exists a sequence of
points~$y_k\geq b_2=m_2$ with~$
y_k\to+\infty$ as~$k\to +\infty$, such that~$
y_k$ is a clean point for~$ Q_{\eta,\mu}^\star$ and~$
Q_{\eta,\mu}^\star(y_k)\in \overline{B_\rho(\zeta_2)}$. Then,
recalling~\eqref{cleanpointx_*} and~\eqref{cleanpointx_*bis},
by~\eqref{STIMA:AA1} and~\eqref{POT:GROW2}, we have that 
\begin{equation*}\int_{x_*}^{y_k} \big[ a(x+\theta)-a(x) \big]\,W(Q_{\eta,\mu}^\star(x))\,dx\leq\diamondsuit.
\end{equation*}
On that account, sending $k\to +\infty$, we obtain that 
\begin{equation}\label{smalnessexthmsfterm2}\int_{m_1+\omega}^{+\infty} \big[ a(x+\theta)-a(x) \big]\,W(Q_{\eta,\mu}^\star(x))\,dx\leq\diamondsuit.
\end{equation}
On the other hand, by arguing as in~\cite{MR3594365}, we have that 
\begin{equation}\label{smalnessexthmsbeforem1}\int_{-\infty}^{m_1-\omega}  \big[ a(x+\theta)-a(x) \big]\,W(Q_{\eta,\mu}^\star(x))\,dx\leq \diamondsuit.
\end{equation}
By plugging  \eqref{pricipalintextthm},  \eqref{smalnessexthmsfterm2}
and~\eqref{smalnessexthmsbeforem1} into~\eqref{9ytguijh}, we conclude that
$$0\leq-\hat{\gamma}+ \diamondsuit+{\const\mu}.$$
Hence there exists a structural constant~$\mu_0>0$
such that, if~$\mu\in[0,\mu_0]$, then
$$0\leq-\frac{\hat{\gamma}}2+ \diamondsuit.$$
The latter inequality is negative for~$\rho$ sufficiently small, and so
we have obtained the desired contradiction.  This proves~\eqref{LE:X1-FREE}. 
\end{proof}

\section{Vanishing viscosity method: sending~$\eta\searrow0$}
\label{S023fhYU0193933}

Now we consider the free minimizer constructed in Proposition~\ref{FREE-eta}
and we send~$\eta\searrow0$. The uniform estimates in~\eqref{LE:X3-FREE},
\eqref{LE:X4-FREE}, \eqref{LE:X5-FREE} and~\eqref{LE:X6-FREE}
will allow us to pass to the limit and obtain
a free minimizer, hence a solution, of a $\mu$-penalized
nonlocal problem (then, in Section~\ref{UPAlao},
we will take the limit as~$\mu\searrow0$ of such a penalization).

This perturbative technique may be thought as a nonlocal counterpart of the so-called
vanishing viscosity method for Hamilton-Jacobi equations, in which a small viscosity
term is added as a perturbation to obtain solutions of the original
equation.

To this aim, we consider~$I_\mu$ to be the energy
functional corresponding to the choice~$\eta:=0$ in~\eqref{01edkIETA},
namely
\begin{equation} \label{01edkIETAmu}\begin{split}
I_\mu(Q)\,&:=
\frac\mu2\int_\R
\big|Q(x)-Q_{\zeta_1,\zeta_2}^\sharp(x)\big|^2\,dx
+\int_\R a(x)W\big(Q(x)\big)\,dx
\\ &\qquad+\frac14
\iint_{\R\times \R} \Big( \big| Q(x)-Q(y)\big|^2
-\big| Q_{\zeta_1,\zeta_2}^\sharp(x)-Q_{\zeta_1,\zeta_2}^\sharp(y)\big|^2\Big)
\,K(x-y)\,dx\,dy.\end{split}
\end{equation}
Then, for any~$\eta>0$, we take~$Q_{\eta,\mu}^\star$
to be the free minimizer given by Proposition~\ref{FREE-eta}.
We consider an infinitesimal sequence~$\eta_j\searrow0$ and let~$Q_j^\star:=Q_{\eta_j,{{\mu}}}^\star$
and~$v_j^\star:=
Q_j^\star-Q_{\zeta_1,\zeta_2}^\sharp$.

Since the estimates in~\eqref{LE:X3-FREE},
\eqref{LE:X4-FREE},~\eqref{LE:X5-FREE} and~\eqref{LE:X6-FREE}  are uniform in~$\eta_j$,
up to a subsequence, for a fixed~$\mu>0$
(suitably small to fulfill the assumption of Proposition~\ref{FREE-eta}),
we can assume that~$v_j^\star$ converges to
some~$v^\star_{{\mu}}$ locally uniformly  in~$\R$ and weakly
in the Hilbert space induced by~$[\cdot]_{K,\R\times \R}$.
Then, we set~$Q^\star_{\mu}:=v^\star_{\mu}+Q_{\zeta_1,\zeta_2}^\sharp$.

By taking limits in~\eqref{LE:X1-FREE}, we obtain that
\begin{equation*}
I_{\mu}(Q_{\mu}^\star)\le I_{\mu}(Q) {\mbox{ for all }} Q {\mbox{ s.t. }}
Q-Q_{\zeta_1,\zeta_2}^\sharp\in H^1(\R)
.\end{equation*}
As a consequence, {differentiating the energy functional in~\eqref{01edkIETAmu} we
find that
\begin{equation}\label{9iskd3yydgdgd}
{\mu (Q-Q^\sharp_{\zeta_1,\zeta_2})}+{\mathcal{L}}(Q)(x)+a(x)  W'(Q(x))=0,
\quad {\mbox{ for any }} x\in\R
\end{equation}
in} the distributional sense,
and thus also in the
viscosity sense, due to~\cite{MR3161511}.

Also, by~\eqref{6cashjdcocn}, we conclude that
\begin{equation}\label{6cashjdcocn-2}
Q_{\mu}^\star:\R\to
\big[\min\{\zeta_1,\zeta_2\},\max\{\zeta_1,\zeta_2\}\big].\end{equation}
Moreover,
in view of~\eqref{PASerf6465-1},
\begin{equation}\label{PASerf6465-2}
Q_{\mu}^\star\in\Gamma(b_1,b_2).\end{equation} 
Finally,
from~\eqref{LE:X6-FREE} we obtain
\begin{equation}
\label{LE:X6-FREEmu} 
\|v_{\mu}^\star\|_{C^{0,\alpha}(\R)}\le\kappa \text{ for all }\alpha\in (0,2s).
\end{equation}

\section{Asymptotics of solutions}\label{ASynsdagg}

Here we discuss the spatial
asymptotics of the solutions of~\eqref{EQUAZIONE}.
In particular, we show that a solution which, at infinity, stays ``close''
to an equilibrium must converge to it.

\begin{lemma}\label{LE121}
Let~$\delta_0>0$ be as in~\eqref{AGGchew}.
Let~$b\in\R$, $\zeta\in{\mathcal{Z}}$ and~$Q$ be a bounded
solution of~\eqref{EQUAZIONE} in~$\R$ such
that
\begin{equation}\label{Dfor2aX53-2}
\zeta-\delta_0\le Q(x)\le \zeta\qquad{\mbox{ for all }}x\ge b.
\end{equation}
Then
$$ \lim_{x\to+\infty}Q(x)=\zeta.$$
\end{lemma}

\begin{proof} Let
$$ \lambda:=\liminf_{x\to+\infty}Q(x).$$
In view of~\eqref{Dfor2aX53-2}, we have that~$\lambda\in[\zeta-\delta_0,\zeta]$,
and that the desired result is proven once we show that
\begin{equation*}
\lambda=\zeta.
\end{equation*}
To prove this, suppose by contradiction that
\begin{equation}\label{Pijnsbdf3844}
\lambda\in[\zeta-\delta_0,\zeta).
\end{equation}
We take a diverging sequence~$x_k$ such that
$$ \lim_{k\to+\infty} Q(x_k)=\lambda,$$
and we define~$Q_k(x):=Q(x+x_k)$.

Since~$\|Q\|_{C^{0,\alpha}(\R)}<+\infty$ for some~$\alpha\in(0,1)$, due to Theorem~12.1
in~\cite{MR2494809}, we have that~$Q_k$
is uniformly equicontinuous and therefore there exists a subsequence~$Q_{k_j}$
that converges locally uniformly to some function~$Q_\infty$.
Then, by Corollary~4.7 in~\cite{MR2494809}, we know that~$Q_\infty$
is a viscosity solution of~\eqref{EQUAZIONE} in~$\R$.

Notice also that, for all~$x\in\R$,
$$ Q_\infty(x)=\lim_{j\to+\infty}Q_{k_j}(x)=\lim_{j\to+\infty}Q(x+x_{k_j})\ge
\liminf_{y\to+\infty}Q(y)=\lambda$$
and
$$ Q_\infty(0)=\lim_{j\to+\infty}Q_{k_j}(0)=\lim_{j\to+\infty}Q(x_{k_j})=
\lambda,$$
from which
\begin{equation}\label{tgvnlwbel}
Q_\infty(x)\ge\lambda=Q_\infty(0)\qquad{\mbox{ for all }}x\in\R.
\end{equation}
Let
$$ w(x):=\begin{cases}
\lambda & {\mbox{ in }} (-1,1),\\
Q_\infty(x)& {\mbox{ in }} \R\setminus(-1,1).
\end{cases}$$
We stress that~$w$ touches~$Q_\infty$ from below at the origin,
in view of~\eqref{tgvnlwbel}, and consequently
\begin{equation}\label{5tgsdd723747566}
{\mathcal{L}}w(0)+a(0)\,W'(\lambda)\ge0.
\end{equation}
On the other hand, by~\eqref{LOPERATO} and~\eqref{tgvnlwbel},
$$ {\mathcal{L}} w(0)= {\rm P.V.}
\int_{\R} \big( \lambda-w(y)\big)\,K(x-y)\,dy=
\int_{\R\setminus(-1,1)} \big( \lambda-Q_\infty(y)\big)\,K(x-y)\,dy\le0.$$
{F}rom this and \eqref{5tgsdd723747566}, we conclude that~$W'(\lambda)\ge0$.
This is in contradiction with~\eqref{AGGchew}, in light of~\eqref{Pijnsbdf3844},
and hence the proof is complete.
\end{proof}

\section{Sending~$\mu\searrow0$ and proof of Theorem~\ref{MAIN}}\label{UPAlao}

We are now ready to complete the proof of Theorem~\ref{MAIN},
by taking the limit as~$\mu\searrow0$ of the penalized solution
constructed in Section~\ref{S023fhYU0193933}.
To this end, we assume, for concreteness, that~$\zeta_2>\zeta_1$.
In light of~\eqref{LE:X6-FREEmu},
up to a subsequence, we can suppose that~$v^\star_\mu$
converges locally uniformly to a function~$v^\star$.
We set~$Q^\star:=v^\star+Q_{\zeta_1,\zeta_2}^\sharp$.
Notice that, by~\eqref{6cashjdcocn-2},
\begin{equation}\label{6cashjdcocn-3}
Q^\star:\R\to
\big[ \zeta_1,\zeta_2\big].\end{equation}
Furthermore,
by Corollary~4.7 in~\cite{MR2494809}, we can also pass
equation~\eqref{9iskd3yydgdgd} to the limit,
thus obtaining that~$Q^\star$ is a solution of~\eqref{EQUAZIONE}.

We also remark that, by~\eqref{PASerf6465-2},
we obtain that~$
Q^\star\in\Gamma(b_1,b_2)$. In particular, 
recalling~\eqref{psi}, \eqref{Psidef} and~\eqref{Gamma},
we have that, for all~$x\ge b_2+1$,
$$Q^\star(x)\ge\Psi(x)=\psi(x)=\zeta_2-r.
$$
Thanks to
this and~\eqref{6cashjdcocn-3}, we are in the position of using Lemma~\ref{LE121}
and thereby conclude that
\begin{equation}\label{SFigbf238457tqe}
\lim_{x\to+\infty}Q^\star(x)=\zeta_2.
\end{equation}
Similarly, one proves that
$$ \lim_{x\to-\infty}Q^\star(x)=\zeta_1.$$
{F}rom this and~\eqref{SFigbf238457tqe},
one obtains~\eqref{0LIMI}, as 
desired.
The proof of Theorem~\ref{MAIN} is thus completed.

\begin{appendix}

\section{Discontinuity and oscillatory behavior at infinity for functions
in Sobolev spaces with low fractional exponents}\label{APPB}

We recall here that functions belonging to the fractional
Sobolev space~$H^s(\R)$ with~$s\in\left(0,\frac12\right)$ are not necessarily continuous,
and they do not need to converge to zero at infinity.

To construct a simple example, let~$\varphi\in C^\infty_0\big(
\R,\,[0,1]\big)$ with~$\varphi(0)=1$. Given a sequence~$b_k$,
let
\begin{equation}\label{01230123401234511} \varphi_{b_k}(x):=\varphi\left( e^k(x-b_k)\right).\end{equation}
Then
\begin{eqnarray*}&&
\|\varphi_{b_k}\|_{L^2(\R)}=\sqrt{
\int_\R \left|\varphi\left( e^k(x-b_k)\right)\right|^2\,dx }
=e^{-\frac{k}2}
\sqrt{\int_\R \left|\varphi\left( X\right)\right|^2\,dX}
=\const e^{-\frac{k}2}\\
{\mbox{and }}&&
[ \varphi_{b_k}]_{H^s(\R)}=\sqrt{\iint_{\R\times\R}
\frac{\left| 
\varphi\left( e^k(x-b_k)\right)-\varphi\left( e^k(y-b_k)\right)
\right|^2}{|x-y|^{1+2s}}\,dx\,dy
}\\&&\qquad\qquad\qquad
=e^{-\frac{(1-2s)k}{2}}\sqrt{\iint_{\R\times\R}
\frac{\left| 
\varphi\left(X\right)-\varphi\left( Y\right)
\right|^2}{|X-Y|^{1+2s}}\,dX\,dY
}=\const e^{-\frac{(1-2s)k}{2}}.
\end{eqnarray*}
We now consider the superposition of the functions~$\varphi_{b_k}$
with the choices~$b_k:=k$ and~$b_k:=1/k$.
Namely, if we set
$$ \Phi(x):=\sum_{k=1}^{+\infty} \varphi_{1/k}(x)+
\sum_{k=1}^{+\infty} \varphi_{k}(x),$$
when~$s\in\left(0,\frac12\right)$ we have that
\begin{eqnarray*}\| \Phi\|_{H^s(\R)}&\le&
\sum_{k=1}^{+\infty} \|\varphi_{1/k}\|_{H^s(\R)}+
\sum_{k=1}^{+\infty} \|\varphi_{k}\|_{H^s(\R)}\\&
\le& \const\sum_{k=1}^{+\infty} \left(
e^{-\frac{k}{2}}+e^{-\frac{(1-2s)k}{2}}\right)\\&\le&\const.\end{eqnarray*}
Nevertheless~$\Phi$ is not continuous at the origin, and
$$ \limsup_{x\to+\infty}\Phi(x)>0=
\liminf_{x\to+\infty}\Phi(x).$$
\medskip

The case of~$H^{1/2}(\R)$ is slightly more delicate,
since simple examples based on scaling, such as the one
provided in~\eqref{01230123401234511}, do not work in this case
(and, in fact, functions in~$H^{1/2}(\R)$
have nicer properties in terms of topology than those
in~$H^{s}(\R)$ with~$s\in\left(0,\frac12\right)$,
see e.g.~\cite{MR1354598}). Nevertheless,
also functions in~$H^{1/2}(\R)$ are not necessarily continuous
and they do not necessarily converge to zero at infinity.
To construct an example of these behaviors, 
as depicted in Figure~\ref{SPIKES}, we 
consider the function
$$ \R^2\ni X\mapsto \psi(X):=\left\{
\begin{matrix}
\log (1-\log|X|) & {\mbox{ if }} X\in B_1\setminus\{0\},\\
0 & {\mbox{ otherwise. }}
\end{matrix}
\right. $$
We claim that
\begin{equation}\label{FUO}
\psi\in H^1(\R^2).
\end{equation}
To check this, we notice that
\begin{equation}\label{02w-1-34567}{\mbox{$\psi$ is supported in~$B_1$,}}\end{equation}
where we have that
$$ |\nabla\psi(X)|=
\frac1{|X|\,\big(1-\log|X|\big)}.$$
Therefore, using polar coordinates and the change of variable~$t:=-\log\rho$,
we find that\begin{eqnarray*}&&
[\psi]^2_{H^1(\R)}=\int_{B_1}
\frac1{|X|^2\,\big(1-\log|X|\big)^2}\,dX
=2\pi\int_{0}^1
\frac1{\rho\,\big(1-\log\rho\big)^2}\,d\rho\\&&\qquad\qquad\qquad=
2\pi\int_{0}^{+\infty}
\frac1{(1+t)^2}\,dt<+\infty
.\end{eqnarray*}
This, together with~\eqref{02w-1-34567}
and the Poincar\'e inequality, proves~\eqref{FUO}.

\begin{figure}
    \centering
    \includegraphics[width=7.5cm]{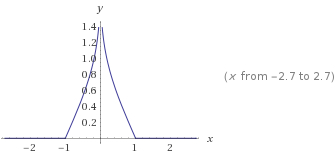}\qquad
    \includegraphics[width=7.5cm]{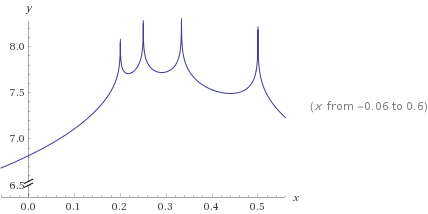}
    \caption{\it {{The function~$\bar\psi$ and sketch of the construction
of the function~$\Psi$.}}}
    \label{SPIKES}
\end{figure}

Then, from~\eqref{FUO}
and the Trace Theorem (see e.g. formula~(3.19)
in~\cite{guida}), we obtain that
\begin{equation}\label{FUO2}
{\mbox{the function $\R\ni x\mapsto\bar\psi(x):=\psi(x,0)$ belongs to~$
H^{1/2}(\R)$.}}
\end{equation}
Now we define the sequence of functions,
for~$k\in\Z$ and~$X=(x,y)\in\R\times\R$,
$$ \psi_k(X)=\psi_k(x,y):=e^{-|k|}\bar\psi\big( e^{|k|}(x-e^k)\big).$$
Then, in view of~\eqref{FUO2} we have that
\begin{eqnarray*}
&& \|\psi_k\|_{L^2(\R)}=e^{-|k|}
\sqrt{ \int_\R \left|\bar\psi\big( e^{|k|}(x-e^k)\big)\right|^2\,dx}
=e^{-\frac{3|k|}2}
\sqrt{\int_{\R} \left|\bar\psi(\eta)\right|^2\,d\eta}\\&&\qquad=
e^{-\frac{3|k|}2}\,\|\bar\psi\|_{L^2(\R)}
=\const e^{-\frac{3|k|}2}\\
{\mbox{and }}&& 
[\psi_k]_{H^{1/2}(\R)}=e^{-|k|}\sqrt{\iint_{\R\times\R}
\frac{\left|\bar\psi\big( e^{|k|}(x-e^k)\big)-
\bar\psi\big( e^{|k|}(y-e^k)\big)\right|^2
}{|\bar x-\bar y|^2}\,d x\,d y}\\&&\qquad=
e^{-|k|} \sqrt{
\iint_{\R\times\R}
\frac{
| \bar\psi( \eta)-\bar\psi(\xi) |^2}{|\eta-\xi|^2} \,d\eta\,d\xi} 
=e^{-|k|}\,[\bar\psi]_{H^{1/2}(\R)}=\const e^{-|k|}
.\end{eqnarray*}
Consequently, if we set
$$ \R\mapsto\Psi(x):=\sum_{k\in\Z} \psi_k(x),$$
it follows that~$\Psi$ is not continuous (and not even locally bounded)
and it does not go to zero at infinity, but it belongs to~$H^{1/2}(\R)$
since
\begin{eqnarray*}&&
\|\Psi\|_{H^{1/2}(\R)}\le\sum_{k\in\Z} \|\psi_k\|_{H^{1/2}(\R)}
\le\const\sum_{k\in\Z}\left( \const e^{-\frac{3|k|}2}+\const e^{-|k|}\right)
\le \const.
\end{eqnarray*}

\end{appendix}

\section*{Acknowledgments}

This work has been
carried out during a very pleasant visit of the first and third authors
to the University of Texas.
Supported by the Australian Research Council Discovery
Project grant ``N.E.W. {\it Nonlocal
Equations at Work}''
and by the DECRA Project grant ``{\it Partial differential equations,
 free boundaries and applications}''. The authors are members of~G.N.A.M.P.A.--I.N.d.A.M.
We thank the anonymous Referees for their very useful remarks.

\vfill
\end{document}